\documentclass[preprint,12pt,3p]{elsarticle}

\usepackage{amssymb}
\usepackage{amsthm}
\usepackage{graphicx}
\usepackage{amsmath}
\usepackage{hyperref}
\usepackage{graphicx}

\usepackage{pgf,tikz}

\usetikzlibrary{snakes}
\tikzstyle{printersafe}=[snake=snake,segment amplitude=0 pt]

\usetikzlibrary{arrows}
\usetikzlibrary{shapes}
\usepackage{multirow}
\usepackage{latexsym}
\usepackage{setspace}
\PassOptionsToPackage{boxed,section}{algorithm}
\usepackage{algorithm}
\usepackage{algorithmic}
\usepackage{enumerate}
\usepackage{amsmath}			% special AMS math definitions
\usepackage{amssymb}			% special AMS math symbols
\usepackage{url}
\usepackage{setspace}
\usepackage{tikz}
\usetikzlibrary{arrows}
\usetikzlibrary{shapes}
\usetikzlibrary{decorations.markings}
\usepackage{geometry}
\usepackage{bbm}

\newtheorem{proposition}{\em Proposition}

\newtheorem{theorem}{\em Theorem}
\newtheorem{conjecture}{\em Conjecture}

\newtheorem{lemma}{\em Lemma}
\newtheorem{claim}{\em Claim}
\newtheorem{claimprime}{\em Claim}

\newtheorem{remark}{\em Remark}
\newtheorem{corollary}{\em Corollary}

\newtheorem{example}{\em Example}

\journal{Latex Templates}

\begin{document}

\begin{frontmatter}

\title{Expanding vertices to triangles in cubic graphs}
%\tnotetext[label0]{The work on this paper has been supported by the National Science Foundation through Grant CCF-2008422}

\author[label2]{Giuseppe Mazzuoccolo}
\address[label2]{Department of Physics, Computer Science and Mathematics\\
University of Modena and Reggio Emilia\\
Via Campi 213/a, 41125 Modena, Italy}
\ead{giuseppe.mazzuoccolo@unimore.it}

\author[label1]{Vahan Mkrtchyan}
\address[label1]{Department of Mathematical Sciences\\
Purdue University Fort Wayne (starting Fall 2025)\\
Fort Wayne, IN, USA-46805}
%\address[label2]{Address Two\fnref{label4}}

%\cortext[cor1]{I am corresponding author}
%\fntext[label3]{I also want to inform about\ldots}
%\fntext[label4]{Small city}
%
\ead{vahan.mkrtchyan@gssi.it}
%\ead[url]{author-one-homepage.com}
%
%\author[label5]{Author Two}
%\address[label5]{Some University}
%\ead{author.two@mail.com}
%
%\author[label1,label5]{Author Three}
%\ead{author.three@mail.com}

\begin{abstract}
Contraction of triangles is a standard operation in the study of cubic graphs, as it reduces the order of the graph while typically preserving many of its properties. In this paper, we investigate the converse problem, wherein certain vertices of cubic graphs are expanded into triangles to achieve a desired property. We first focus on bridgeless cubic graphs and define the  parameter $T(G)$ as the minimum number of vertices that need to be expanded into triangles so that the resulting cubic graph can be covered with four perfect matchings. We relate this parameter to the concept of shortest cycle cover. Furthermore, we show that if $5$-Cycle Double Cover Conejcture holds true, then $T(G)\leq \frac{2}{5} |V(G)|$. We conjecture a tighter bound, $T(G)\leq \frac{1}{10}|V(G)|$, which is optimal for the Petersen graph, and show that this bound follows from major conjectures like the Petersen Coloring Conjecture.
In the second part of the paper, we introduce the parameter $t(G)$ as the minimum number of vertex expansions needed for the graph to admit a perfect matching. We prove a Gallai type identity: $t(G)+\ell(G)=|V(G)|$, where $\ell(G)$ is the number of edges in a largest even subgraph of $G$. Then we prove the general upper bound $t(G)< \frac{1}{4}|V(G)|$ for cubic graphs, and $t(G)< \frac{1}{6}|V(G)|$ for cubic graphs without parallel edges. We provide examples showing that these bounds are asymptotically tight. The paper concludes with a discussion of the computational complexity of determining these parameters.
\end{abstract}
%and the existence of a similar algorithm for the largest cardinality matching problem. 

\begin{keyword}
Matching \sep cubic graph \sep perfect matching \sep 2-factor \sep triangle.
\MSC[2020] 05C70 \sep 05C15.
\end{keyword}

\end{frontmatter}

% \linenumbers


\section{Introduction}
\label{IntroSection}

Many important conjectures in graph theory, such as the Cycle Double Cover Conjecture or 5-Flow Conjecture, can be reduced to the case of cubic graphs, graphs in which every vertex has degree three. We mean that proving a conjecture for cubic graphs is enough to establish it in its general formulation. When working with the conjectures mentioned above, triangles are often considered a trivial case, since their contraction reduces the size and complexity of the graph. In this paper, we consider the reverse approach: instead of contracting triangles, we study what happens when we expand certain vertices of a cubic graph into triangles. Our goal is to understand whether such an expansion can enforce a desired property that the original graph may not have satisfied.

In this paper, we consider finite, undirected graphs that do not contain loops. However, they may contain parallel edges. If they do not contain parallel edges, then we will refer to them as simple graphs. 

For a bridgeless cubic graph $G$, we introduce the parameter $T(G)$ as the minimum number of vertices that we need to expand into triangles so that the edge-set of the resulting cubic graph can be covered with four perfect matchings. In some sense our parameter measures how far is our graph from being coverable with four perfect matchings. Similar measures for other properties of cubic graphs are presented in \cite{steffen:2004} and \cite{fiol}. 
A strong motivation for studying the parameter $T(G)$ is that, if it can be shown that $T(G)$ is finite for every cubic graph, this would imply the $5$-CDC Conjecture (see \cite{HakobyanAMC} and Conjecture \ref{conj:ClawFree4Coverable}).
We relate the parameter $T(G)$ to the length of a shortest cycle cover of the graph $G$. Moreover, we show that if $5$-Cycle Double Cover Conjecture is verified, then $T(G)\leq \frac{2}{5}|V|$ holds. Anyway, this bound does not seem tight. In the paper we offer a conjecture that $T(G)\leq \frac{1}{10}|V|$. To evaluate the plausibility of our conjecture, observe that the bound is asymptotically tight, as evidenced by the Petersen graph, and we demonstrate that it follows as a consequence of several established conjectures.

In Section \ref{CubicResultsSection}, we introduce another parameter $t(G)$, defined as the minimum number of vertices in an arbitrary cubic graph $G$ that need to be expanded into triangles so that the resulting graph admits a perfect matching. In connection with this parameter, we preliminary show that for every cubic graph $G$ the identity $t(G)+\ell(G)=|V(G)|$ holds, where $\ell(G)$ denotes the number of edges in a largest even subgraph of $G$. We then use this equality to show that $t(G)<\frac{1}{4}|V(G)|$ for an arbitrary cubic graph, and $t(G)< \frac{1}{6}|V(G)|$ for simple cubic graphs. These two bounds are complemented by examples demonstrating their tightness.

We conclude the paper in Section \ref{sec:complexity} by discussing the computational complexity of determining both parameters.

\section{Notations and auxiliary results}
\label{AuxSection}
This section is devoted to introducing the notation and definitions that will be used throughout the paper. We also recall all known results that will play a role in the forthcoming sections. Non-defined terms and concepts can be found, for instance, in \cite{west:1996}.
  
If $G$ is a graph, then let $V(G)$ and $E(G)$ denote the sets of vertices and edges of $G$, respectively. When the graph is clear from context, we simply write $V$ and $E$.

%A matching in a graph $G$ is a subset $M$ of edges such that no two edges of $M$ share a vertex. A matching $M$ is perfect if every vertex of the graph is incident to an edge from $M$. For a graph $G$ and a vertex $v$ let $\partial_{G}(v)$ be the set of edges of $G$ that are incident to $v$ in $G$. As usual, the cardinality of $\partial_G(v)$ will be called the degree of the vertex $v$ in $G$. A graph $G$ is $k$-regular if every vertex of $G$ is of degree $k$. A graph is cubic if it is $3$-regular. Let $k\geq 1$. A $k$-factor of a graph $G$ is a spanning $k$-regular subgraph of $G$. Note that if $K$ is a 1-factor of $G$, then $E(K)$ is a perfect matching in $G$. If $G$ is a cubic graph then $F$ is a perfect matching in $G$ if and only if $G-F$ is a $2$-factor in $G$. This $2$-factor will be called a complementary $2$-factor of $F$ in $G$. For a perfect matching $F$ of a cubic graph $G$, its complementary $2$-factor will be denoted by $\overline{F}$.

If $G$ is a graph, then a block of $G$ is an inclusion-wise maximal $2$-connected subgraph of $G$. If $B$ is a block in $G$ and it contains at most one cut-vertex, then $B$ will be called an \emph{end-block}. Let $G$ be a cubic graph with bridges and let $B$ be an end-block of $G$, there is a unique cut-vertex $y$ of $G$ in $B$ that is adjacent to a unique vertex $x$ lying outside $B$. We will refer along the paper to $x$ as the root of $B$. 

In Section \ref{CubicResultsSection} a special role will be played by the graph with three vertices obtained from a triangle by duplicating one of its edges (see the three end-blocks of $S_{10}$ from Figure \ref{fig:SylvGraph}). From now on, we will denote by $W$ such a graph. Similarly, $W'$ will denote the unique graph obtained from a complete graph $K_4$ by subdividing one of its edges exactly once.

If $G$ is a cubic graph and $e=uv$ is an edge in it, then the following operation will be relevant: subdivide the edge $e=uv$ with a new vertex $w_e$, add a copy of $W$ and join the unique degree-two vertex of $W$ to $w_e$ (Figure \ref{fig:OperationGraphs}). Note that the resulting graph is cubic, the added copy of $W$ is an end-block in it whose root is $w_e$. We often say that {\it we subdivide $e$ and attach a copy of $W$ to it.} We will define a similar operation when instead of $W$ we work with $W'$. 

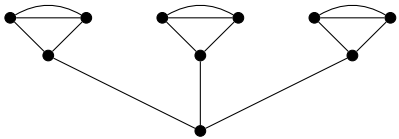
\begin{figure}[ht]
  
  \begin{center}
	
	\tikzstyle{every node}=[circle, draw, fill=black!50,
                        inner sep=0pt, minimum width=4pt]
	
		\begin{tikzpicture}%[-,>=stealth',shorten >=1pt,auto,node distance=3cm,
		%	thick,main node/.style={circle,fill=blue!20,draw,font=\sffamily\Large\bfseries}]
			% \begin{tikzpicture}
																							
			\node[circle,fill=black,draw] at (-5.5,-1) (n1) {};
			%\draw[black,fill=black,thick] (-5.5,-1) circle [radius=0.1cm] ;
			%\node at (-5, 4.35) (l1) {$r$ $(2,0)$};
			
			%\draw[fill=black,thick] (-5.5,-1) circle (2pt);
			%\node at (-5.5, -1) (n1){};
			
			%\draw[fill=black] (18:1cm) circle (2pt);
			
			%\draw[fill=black] (240:1cm) circle (2pt);
			
			%\draw[fill=black] ((-5.5,-1):1cm) circle (2pt);

			\node[circle,fill=black,draw] at (-6, -0.5) (n2) {};
			% \node at (-3.3, 3.35) (l2) {$(0,1)$};
																								
			\node[circle,fill=black,draw] at (-5,-0.5) (n3) {};
			%  \node at (-6.8, 3.35) (l2) {$(0,1)$};
																								
			\node[circle,fill=black,draw] at (-3.5,-1) (n4) {};
			%\node at (-5, 4.35) (l1) {$r$ $(2,0)$};
																								
			\node[circle,fill=black,draw] at (-4, -0.5) (n5) {};
			% \node at (-3.3, 3.35) (l2) {$(0,1)$};
																								
			\node[circle,fill=black,draw] at (-3,-0.5) (n6) {};
			%  \node at (-6.8, 3.35) (l2) {$(0,1)$};
																								
			\node[circle,fill=black,draw] at (-1.5,-1) (n7) {};
			%\node at (-5, 4.35) (l1) {$r$ $(2,0)$};
																								
			\node[circle,fill=black,draw] at (-2, -0.5) (n8) {};
			% \node at (-3.3, 3.35) (l2) {$(0,1)$};
																								
			\node[circle,fill=black,draw] at (-1,-0.5) (n9) {};
			%  \node at (-6.8, 3.35) (l2) {$(0,1)$};
																								
			\node[circle,fill=black,draw] at (-3.5,-2) (n10) {};

			\path[every node]
			(n1) edge  (n2)

			edge  (n3)
			edge (n10)
			%edge [bend right] (n2)
			%edge [bend left] (n2)
			%edge [bend right] (n3)
																								   	
			(n2) edge (n3)
			edge [bend left] (n3)
			%edge [bend left] (n3)
			%edge [bend right] (n3)
																								       
			(n3) 
			%  edge [bend left] (n1)
			% edge [bend right] (n1) ;
			(n4) edge (n5)
			edge (n6)
			edge (n10)
																								    
			(n5) edge (n6)
			edge [bend left] (n6)
			(n6)
																								   
			(n7) edge (n8)
			edge (n9)
			edge (n10)
																								    
			(n8) edge (n9)
			edge [bend left] (n9)
																								  
			;
		\end{tikzpicture}
																
	\end{center}
								
	\caption{The graph $S_{10}$.}
	\label{fig:SylvGraph}

\end{figure}

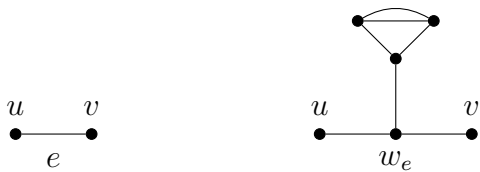
\begin{figure}[ht]
  
  \begin{center}

		\begin{tikzpicture}%[-,>=stealth',shorten >=1pt,auto,node distance=3cm,
		%	thick,main node/.style={circle,fill=blue!20,draw,font=\sffamily\Large\bfseries}]
			% \begin{tikzpicture}

            \node at (-10, -2.35) {$e$};

            \node at (-10.5, -1.65) {$u$};
            \node at (-9.5, -1.65) {$v$};

            \node at (-6.5, -1.65) {$u$};
            \node at (-4.5, -1.65) {$v$};
            \node at (-5.5, -2.35) {$w_e$};

            \tikzstyle{every node}=[circle, draw, fill=black!50,
                        inner sep=0pt, minimum width=4pt]
																							
			\node[circle,fill=black,draw] at (-5.5,-1) (n1) {};
			%\draw[black,fill=black,thick] (-5.5,-1) circle [radius=0.1cm] ;
			%\node at (-5, 4.35) (l1) {$r$ $(2,0)$};
			
			%\draw[fill=black,thick] (-5.5,-1) circle (2pt);
			%\node at (-5.5, -1) (n1){};
			
			%\draw[fill=black] (18:1cm) circle (2pt);
			
			%\draw[fill=black] (240:1cm) circle (2pt);
			
			%\draw[fill=black] ((-5.5,-1):1cm) circle (2pt);

			\node[circle,fill=black,draw] at (-6, -0.5) (n2) {};
			% \node at (-3.3, 3.35) (l2) {$(0,1)$};
																								
			\node[circle,fill=black,draw] at (-5,-0.5) (n3) {};
			%  \node at (-6.8, 3.35) (l2) {$(0,1)$};

                \node[circle,fill=black,draw] at (-5.5,-2) (n4) {};

                \node[circle,fill=black,draw] at (-6.5,-2) (n5) {};
                \node[circle,fill=black,draw] at (-4.5,-2) (n6) {};

                 \node[circle,fill=black,draw] at (-10.5,-2) (n7) {};
                \node[circle,fill=black,draw] at (-9.5,-2) (n8) {};

			\path[every node]
			(n1) edge  (n2)
                edge (n4)

			edge  (n3)
			
			%edge [bend right] (n2)
			%edge [bend left] (n2)
			%edge [bend right] (n3)
																								   	
			(n2) edge (n3)
			edge [bend left] (n3)
			%edge [bend left] (n3)
			%edge [bend right] (n3)
                (n4) edge (n5)
                (n4) edge (n6)

                (n7) edge (n8)

			;
		\end{tikzpicture}
																
	\end{center}
								
	\caption{Subdividing $e$ and attaching a copy of $W$ to it.}
	\label{fig:OperationGraphs}

\end{figure}

%If $n\geq 1$, then let $K_n$ denote the unique graph on $n$ vertices where every pair of vertices is an edge in it. Such graph is called complete and we will denote it by $K_n$ (Figure \ref{fig:K4}). A graph $G$ is bipartite, if $V(G)$ can be partitioned into two sets $V_1$ and $V_2$, such that every edge of $G$ joins a vertex from $V_1$ to $V_2$. A bipartite graph is called complete if every vertex of $V_1$ is joined to every vertex of $V_2$. When $G$ is a complete bipartite graph with $|V_1|=m$ and $|V_2|=n$, then it will be denoted by $K_{m,n}$. The graph $K_{1,3}$ will be called claw.

%\begin{figure}[!htbp]
%\begin{center}
%\begin{tikzpicture}[scale=0.35]

% \tikzstyle{every node}=[circle, draw, fill=black!,
 %                       inner sep=0pt, minimum width=4pt]
  
  %\node[circle,fill=black,draw] at (-2,2) (n2) {};
 % \node at (-2, 2.55) (v1) {$v_2$};

   %\node[circle,fill=black,draw] at (-2,-2) (n3) {};
 %  \node at (-2, -2.55) (v1) {$v_3$};

    %\node[circle,fill=black,draw] at (-6,2) (n4) {};
 %   \node at (-6, 2.55) (v1) {$v_4$};
    
    %\node[circle,fill=black,draw] at (-6,-2) (n5) {};
  %  \node at (-6, -2.55) (v1) {$v_5$};

% \draw (n2)--(n3);
 
 %\draw (n2)--(n4);
 %\draw (n3)--(n4);
 %\draw (n3)--(n5);
% \draw (n2)--(n5);
 
 %\draw (n4)--(n5);

%\end{tikzpicture}
%\end{center}
%\caption{The graph $K_4$.} \label{fig:K4}
%\end{figure}

As usual, a circuit in a graph $G$ is a connected 2-regular subgraph. If $C$ is a circuit in $G$, we refer to $|E(C)|$ as the length of $C$. If $C$ is a circuit of length three, then we will call it a triangle. Whereas, in this context, we use the term cycle in a broader sense: a subgraph in which every vertex has positive even degree. In the case of cubic graphs, this definition implies that a cycle is a union of circuits, and the length of a cycle is simply the sum of the lengths of its circuits.
A cycle cover $\mathcal{C}$ of $G$ is a list of cycles in $G$ such that every edge of $G$ belongs to at least one element of $\mathcal{C}$. The length of a cycle cover $\mathcal{C}$ of $G$ is the sum of length of all cycles in $\mathcal{C}$. For a bridgeless graph $G$, $scc(G)$ denote the length of a shortest cycle cover of $G$. 
We define an even subgraph of a graph $G$ as a spanning subgraph in which every vertex has even degree. Under this terminology, even subgraphs may include isolated vertices (i.e., vertices of degree zero), in contrast to cycles, which by definition do not. The complement of an even subgraph is referred to as a parity subgraph of $G$.

%A graph $G$ is called a tree, if it connected and contains no circuits. A spanning subgraph $H$ of $G$ is even, if every vertex of $G$ is of even degree in $H$. A subgraph $J$ of $G$ is called a parity subgraph of $G$, if $G\backslash E(J)$ is an even subgraph in $G$. 

In this paper we also need to consider the Petersen Coloring Conjecture of Jaeger and its classical consequences. This conjecture asserts that for every bridgeless cubic graph $G$ its edge-set $E(G)$ can be colored by using as set of colors $E(P_{10})$, where $P_{10}$ is the Petersen graph (Figure \ref{fig:Petersen10}), such that adjacent edges of $G$ receive as colors adjacent edges of $P_{10}$.

   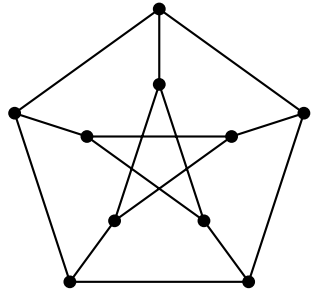
\begin{figure}[ht]
	\begin{center}
	\begin{tikzpicture}[style=thick]
\draw (18:2cm) -- (90:2cm) -- (162:2cm) -- (234:2cm) --
(306:2cm) -- cycle;
\draw (18:1cm) -- (162:1cm) -- (306:1cm) -- (90:1cm) --
(234:1cm) -- cycle;
\foreach \x in {18,90,162,234,306}{
\draw (\x:1cm) -- (\x:2cm);
\draw[fill=black] (\x:2cm) circle (2pt);
\draw[fill=black] (\x:1cm) circle (2pt);
}
\end{tikzpicture}
	\end{center}
	\caption{The Petersen graph $P_{10}$.}\label{fig:Petersen10}
\end{figure}

Here, we introduce such a conjecture in a more formal way. 
Let $G$ and $H$ be two cubic graphs. If there is a mapping $\phi:E(G)\rightarrow E(H)$, such that for each $v\in V(G)$ there is $w\in V(H)$ such that $\phi(\partial_{G}(v)) = \partial_{H}(w)$, then $\phi$ is called an $H$-coloring of $G$. If $G$ admits an $H$-coloring, then we will write $H
\prec G$. It can be easily seen that if $H\prec G$ and $K\prec H$, then $K\prec G$. In other words, $\prec$ is a transitive relation defined on the set of cubic graphs. 

\begin{example}
    If $G$ is the complete bipartite graph $K_{3,3}$ and $H$ is the complete graph $K_4$, then Figure \ref{fig:HcoloringG} shows an example of an $H$-coloring of $G$. Here $V(H)=\{1,2,3,4\}$ and $E(H)=\{a_1, a_2, a_3, a_4, a_5, a_6\}$. Figure \ref{fig:HcoloringG} shows the colors of edges of $G$ with the edges of $H$.

\begin{figure}[ht]
		\begin{center}
			\begin{tikzpicture}[scale=0.85]
			
%%%%%% some constants

\def \r {0.1}
\def \radius {\r cm}
\def \c {\r}

%%%% Description of the graph H
\node at (1,-2) {$H$};

\draw[fill=black]  (0,0) circle (\radius);
%\node at (0,0) {$1$};

\draw[fill=black]  (0,2) circle (\radius);
%\node at (0,2) {$2$};

\draw[fill=black]  (2,2) circle (\radius);
%\node at (2,2) {$3$};

\draw[fill=black]  (2,0) circle (\radius);
%\node at (2,0) {$4$};

\draw[-, thick] (0,\r)--(0,2-\r); 
\node at (-0.3,1) {$a_1$};

\draw[-, thick] (\r,2)--(2-\r,2);
\node at (1,2.3) {$a_2$};

\draw[-, thick] (2,2-\r)--(2,\r);
\node at (2.3,1) {$a_3$};

\draw[-, thick] (2-\r,0)--(\r,0);
\node at (1,-0.3) {$a_4$};

\draw[-, thick] (\r/2,\r/2)--(2-\r/2,2-\r/2);
%\node at (2*\r,2*\r+\c) {$a_5$};
\node at (0.75,0.4) {$a_5$};

\draw[-, thick] (2-\r/2,\r/2)--(\r/2,2-\r/2);
%\node at (2-2*\r,2*\r-\c) {$a_6$};
\node at (0.75,1.55) {$a_6$};
%%%%%%%%%%%%%%%%%%%%%%%%%%%%%%%%%%%%%%%%%%
%\node at (3,1) {$H\prec G$};
%%%%%%%%%%%%%%%%%%%%%%%%%%%%%%%%%%%%%%%%%%

%%%% Description of the graph G
\node at (6,-2) {$G$};

\draw[fill=black] (4,0) circle (\radius);
%\node at (4,0) {$1$};

\draw[fill=black] (8,0) circle (\radius);
%\node at (8,0) {$3$};

\draw[fill=black] (6,-1) circle (\radius);
%\node at (6,-2) {$4$};

\draw[fill=black] (4,2) circle (\radius);
%\node at (4,2) {$1$};

\draw[fill=black] (8,2) circle (\radius);
%\node at (8,2) {$3$};

\draw[fill=black] (6,3) circle (\radius);
%\node at (6,4) {$4$};

%%%%%% the edges of the 6-cycle

\draw[-, thick] (4,0)--(4,2);
\node at (3.7,1) {$a_1$};

\draw[-, thick] (8,\r)--(8,2-\r);
\node at (8.3,1) {$a_2$};

\draw[-, thick] (4+\r/2,-\r/2)--(6,-1);
\node at (4.7,-0.7) {$a_4$};

\draw[-, thick] (8-\r/2,-\r/2)--(6,-1);
\node at (7.3,-0.7) {$a_3$};

\draw[-, thick] (4+\r/2,2+\r/2)--(6,3);
\node at (4.7,2.7) {$a_4$};

\draw[-, thick] (8-\r/2,2+\r/2)--(6,3);
\node at (7.3,2.7) {$a_3$};

%%%%%% the diagonals

\draw[-, thick] (6,-1)--(6,3);
\node at (6.3,-0.5) {$a_6$};
\node at (6.3,2.5) {$a_6$};

\draw[-, thick] (4+\r/2,\r/2)--(8-\r/2,2-\r/2);
\node at (4.8,0.2) {$a_5$};
\node at (7.15,1.8) {$a_5$};

\draw[-, thick] (4+\r/2,2-\r/2)--(8-\r/2,\r/2);
\node at (7.15,0.2) {$a_5$};
\node at (4.8,1.8) {$a_5$};

			\end{tikzpicture}

		\end{center}
		
		\caption{An example of an $H$-coloring of $G$.}\label{fig:HcoloringG}
	\end{figure}
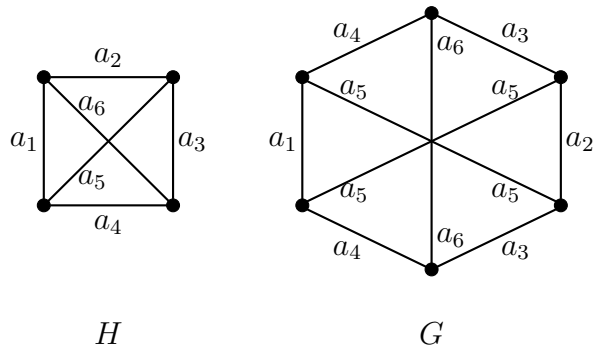

\end{example}

\begin{conjecture}\label{conj:P10conj} (Jaeger, 1988 \cite{Jaeger1979,Jaeger1985,Jaeger1988}) For any bridgeless cubic
graph $G$, one has $P_{10} \prec G$.
\end{conjecture}
The conjecture is well-known in graph theory. It is considered hard to prove since it implies some other classical conjectures in the field such as Berge-Fulkerson Conjecture (Conjecture \ref{conj:BergeFulkerson} below), Cycle Double Cover Conjecture, $5$-Cycle Double Cover Conjecture (Conjecture \ref{conj:52CDC} below) and the Shortest Cycle Cover Conjecture (Conjecture \ref{conj:75conj} below) (see \cite{Fulkerson,Jaeger1985,Zhang1997}).
 %The conjecture is difficult to prove, since it can be seen that it implies the following two classical conjectures:
\begin{conjecture}\label{conj:BergeFulkerson} (Berge-Fulkerson, 1972 \cite{Fulkerson,Seymour}) Any bridgeless
cubic graph $G$ contains six (not necessarily distinct) perfect matchings
$F_1, \ldots , F_6$ such that any edge of $G$ belongs to exactly two of them.
\end{conjecture}

\begin{conjecture}\label{conj:52CDC}
(5-Cycle Double Cover Conjecture, \cite{Celmins1984,Preiss1981}) Any bridgeless
graph $G$ contains five cycles such that any
edge of $G$ belongs to exactly
two of them.
\end{conjecture}

\begin{conjecture}
\label{conj:75conj} (Shortest Cycle Cover Conjecture, \cite{JT1992,Zhang1997,CQZhangBook2}) Let $G$ be a  bridgeless graph. Then, $scc(G)\leq \frac{7}{5}|E(G)|$ holds.
\end{conjecture}

%It can be shown that the Petersen graph is the only 2-edge-connected cubic graph that can color all bridgeless cubic graphs \cite{Mkrt2013}. Note that \cite{Mazz11} proves that Conjecture \ref{conj:BergeFulkerson} is equivalent to proving that the edge-set of all bridgeless cubic graphs can be covered with five perfect matchings. Moreover, see the recent paper \cite{KMZ22}, which shows that every bridgeless cubic graph $G$ has a pair of perfect matchings $F_1$ and $F_2$, such that $G-(F_1\cup F_2)$ is bipartite. Note that this statement is a corollary of Conjecture \ref{conj:BergeFulkerson}, which was conjectured by the first author in \cite{Mazz2013}. 

More conjectures similar to Conjecture \ref{conj:P10conj} can be found in \cite{HakobyanAMC,Mkrt2013}. Recent results about $H$-colorings can be found in \cite{rGraphHColorings} and \cite{HcoloringsRegulars}, where the graphs under consideration are regular, but not necessarily cubic.

We conclude this section with a list of some results that would be used later in the paper.

\begin{proposition}
    \label{prop:31degreeTrees} Let $T$ be a tree in which every vertex is of degree $1$ or $3$. Assume $T$ has $n$ vertices. Let $k_1$ and $k_3$ be the number of vertices with degrees $1$ and $3$, respectively. Then, $k_1=\frac{n}{2}+1$ and $k_3=\frac{n}{2}-1$ hold.
\end{proposition}
%\begin{proof} By definition:
%\[k_1+k_3=n.\]
%On the other hand, since $T$ is a tree, it has $(n-1)$ edges. Thus:
%\[k_1+3k_3=2(n-1)=2n-2.\]
%By solving these two equations with respect to $k_1$ and $k_3$ gives the required statement. The proof is complete.
%\end{proof}

The trees discussed above have played an important role in \cite{samvel:2010,corrigendum}.
A classical theorem by Petersen can be stated as follows.

\begin{theorem}
    \label{thm:Petersen} (Petersen, see \cite{Sch} or Section 3.4 of \cite{Lovasz})
    \begin{enumerate}
        \item [(a)] Let $G$ be a bridgeless cubic graph. Then $G-e_1-e_2$ has a perfect matching for every $e_1, e_2\in E(G)$;
        
        \item [(b)] Let $G$ be a cubic graph with at most two bridges. Then $G$ has a perfect matching.
    \end{enumerate}
\end{theorem}

\begin{corollary}
    \label{cor:Endblock2factor} Let $G$ be a bridgeless graph in which all degrees of vertices are three, except one which is of degree two. Then $G$ contains a $2$-factor. In particular, every end-block of a cubic graph contains a $2$-factor. 
\end{corollary}  

\begin{theorem}
    \label{thm:Steffen4Coverable} (\cite{Steffen2015}) Let $G$ be a bridgeless cubic graph. The edge-set of $G$ can be covered with four perfect matching if and only if 
    \[scc(G)= \frac{4}{3}|E(G)|.\]
\end{theorem}

\begin{theorem}
    \label{thm:ParityGraphs} (\cite{Hou2016,Steffen2015}) Let $G$ be a bridgeless graph. Then $G$ admits a cover with five cycles such that each edge is covered twice, if and only if $G$ can be covered with four parity subgraphs such that every edge is covered at most twice.
\end{theorem}

\section{Expanding Vertices to Reduce the Perfect Matching index}
\label{BridgelessCubicResultsSection}

Let $G$ be a cubic graph and let $U\subseteq V(G)$. The cubic graph obtained from $G$ by expanding the vertices of $U$ to triangles will be denoted by $G_U$. 
For a bridgeless cubic graph $G$, let $T(G)$ be the size of a smallest set $U$ such that the edge-set of the resulting cubic graph $G_U$ can be covered with four perfect matchings. In this section we obtain some bounds on $T(G)$ and we link its value to some well-known conjectures. %We start with the following proposition, in which a subgraph $R$ of a graph $G$ is called a parity subgraph if $G-E(R)$ is an even subgraph in $G$.
The following remark is an immediate consequence of the definition of a parity subgraph.

\begin{remark}
    \label{prop:ParitySubgraphCubic} Let $G$ be a cubic graph, $J$ be a parity subgraph of $G$ and $U$ be the set of degree-three vertices in $J$. Then $J$ is a perfect matching in $G_U$.
\end{remark} A priori it is unclear why starting with any bridgeless cubic graph and expanding some vertices to triangles, we will obtain a bridgeless cubic graph that can be covered with four perfect matchings. However, we proved that 5-Cycle Double Cover Conjecture is equivalent to the following statement.

\begin{conjecture}
    \label{conj:ClawFree4Coverable} (\cite{HakobyanAMC}) Any claw-free bridgeless cubic graph can be covered with four perfect matchings.
\end{conjecture} 
See \cite{HakobyanAMC} for the proof of this equivalence. The proof given there relies on Theorem \ref{thm:ParityGraphs} from \cite{Hou2016}. This means that if 5-Cycle Double Cover Conjecture holds true then $T(G)$ is well-defined and $T(G)\leq |V(G)|$. 

By definition, $T(G)=0$ if and only if the graph $G$ itself admits a cover of the edge-set with at most four perfect matchings. In particular, if $G$ admits a $3$-edge-coloring then $T(G)=0$.

\subsection{Relations between $T(G)$ and $scc(G)$}

Theorem \ref{thm:Steffen4Coverable} suggests a link between $scc(G)$-the length of a shortest cycle cover of $G$ and the parameter $T(G)$. In what follows, we aim to strengthen this connection.

\begin{proposition}
\label{prop:sccT(G)} Let $G$ be a bridgeless cubic graph. Then \[scc(G)\leq \frac{4}{3}|E(G)|+T(G).\]
\end{proposition}

\begin{proof} Let $G$ be any bridgeless cubic graph. If we replace one vertex of $G$ with a triangle, then for the resulting cubic graph $G'$ we will have
%\[scc(G)\leq scc(G')-3,\]
%or
\[scc(G')\geq scc(G)+3.\]
Now, if $U$ is a smallest subset of $V(G)$, such that $G_U$ can be covered with four perfect matchings, then we will have
\[scc(G_U)\geq scc(G)+3|U|=scc(G)+3T(G).\]
By Theorem \ref{thm:Steffen4Coverable}, we have
\[scc(G_U)\leq \frac{4}{3}|E(G_U)|,\]
hence
\begin{align*}
    scc(G) &\leq scc(G_U)-3T(G)\leq \frac{4}{3}|E(G_U)|-3T(G)\\
           &= \frac{4}{3}|E(G)|+4T(G)-3T(G)=  \frac{4}{3}|E(G)|+T(G). 
\end{align*}
The proof is complete.
\end{proof}

We are unable to prove that, in the previous statement, the inequality can be replaced by an equality. However, we suspect this may be the case, at least under the additional assumption that $G$ is $3$-edge-connected, and we propose it as a conjecture.

\begin{conjecture}
    \label{conj:sccT(G)} Let $G$ be a $3$-edge-connected cubic graph. Then \[scc(G)= \frac{4}{3}|E(G)|+T(G).\]
\end{conjecture}

To support our conjecture, we prove that it follows from another well-known conjecture by C.-Q. Zhang (see \cite{Zhang1997}). To state this conjecture, we first define the depth of an edge $e$ in a cycle cover as the number of cycles containing $e$. Accordingly, we define the depth of a cycle cover as the maximum depth among all edges of $G$.

\begin{conjecture}
    \label{conj:12cover} (Conjecture 8.11.6 from \cite{Zhang1997}) Let $G$ be a $3$-edge-connected graph. Then, $G$ admits a shortest cycle cover of depth $2$.
\end{conjecture}

\begin{proposition}
    \label{thm:12coverimpliesequality}
Conjecture \ref{conj:12cover} implies Conjecture \ref{conj:sccT(G)}.
\end{proposition}
\begin{proof}
By Proposition \ref{prop:sccT(G)}, it suffices to prove that for any $3$-edge-connected cubic graph $G$, we have 
\[scc(G)\geq \frac{4}{3}|E(G)|+T(G),\]
or equivalently,
\[T(G)\leq scc(G)-\frac{4}{3}|E(G)|.\]
Consider a shortest cycle cover $\cal C$ of depth $2$ in $G$, meaning that each edge of $G$ belongs to either one or two cycles in $\cal C$. The vertex-set of $G$ is partitioned in two subsets: the set $U$ of vertices incident to three edges of depth $2$ and the one of vertices incident to exactly one edge of depth $2$. Clearly, 
\[ scc(G)=\frac{4}{3}|E(G)|+|U|.\]
%or
%\[|U|=scc(G)-\frac{4}{3}|E(G)|.\]
Next, we expand every vertex $u$ in $U$ into a triangle, thus obtaining a bridgeless cubic graph $G_U$. We then extend every cycle of $\cal C$ passing through $u$ to a cycle in $G_U$ by adding exactly one edge of the corresponding triangle. The resulting set of cycles forms a cycle cover of $G_U$ with length $\frac{4}{3}|E(G_U)|$, since the set of edges covered twice forms a perfect matching of $G_U$, while every other edge is covered once.
Since $G_U$ admits a cycle cover of length $\frac{4}{3}|E(G_U)|$, it follows from Theorem \ref{thm:Steffen4Coverable} that its edge-set can be covered with four perfect matchings. Hence, 
\[T(G)\leq |U|=scc(G)-\frac{4}{3}|E(G)|.\]
%\[scc(G) = \frac{4}{3}|E(G)|+|U| \geq \frac{4}{3}|E(G)|+T(G).\]
   The proof is complete.
\end{proof}

We leave as an open problem determining whether there exists a bridgeless cubic graph such that $scc(G)<\frac{4}{3}|E(G)|+T(G)$.

\subsection{Some Upper Bounds for $T(G)$}

We have already observed that $5$-Cycle Double Cover Conjecture implies $T(G) \leq |V(G)|.$ In this section, we propose some stronger upper bounds for the parameter $T(G)$ in terms of the order of the graph.

\begin{theorem}\label{thm:25|V|UpperBound}
    If Conjecture \ref{conj:52CDC} holds true, then $T(G)\leq \frac{2}{5}|V(G)|$.
\end{theorem}

\begin{proof} We follow the approach of \cite{Hou2016} and \cite{Steffen2015}. Let $G$ be a bridgeless cubic graph. Let $\{C_0,...,C_4\}$ be a 5-CDC of $G$. Observe that \[|C_0|+...+|C_4|=2|E|,\] where $|C_j|$ is the length of $C_j$. We can assume that \[|C_0|\geq \frac{2}{5}|E|.\] Consider the even cover \[\mathcal{C}_0=\{C_0\bigtriangleup C_1,...,C_0\bigtriangleup C_4\}.\] As in \cite{Hou2016}, we have if $e\in C_0$, then $\mathcal{C}_0$ covers $e$ three times, and if $e\notin C_0$, then $\mathcal{C}_0$ covers $e$ two times. Let \[\mathcal{J}_0=\{\overline{C_0\bigtriangleup C_1},...,\overline{C_0\bigtriangleup C_4}\}\] be the set of complements of the subgraphs $C_0\bigtriangleup C_1,...,C_0\bigtriangleup C_4$. Observe that all of them are parity subgraphs. Moreover,  if $e\in C_0$, then $\mathcal{J}_0$ covers $e$ once, and if $e\notin C_0$, then $\mathcal{J}_0$ covers $e$ twice. Observe that if an edge $e$ is covered once in $\mathcal{J}_0$, then no parity subgraph of $\mathcal{J}_0$ has degree three on endpoints of $e$. This means that the number of vertices of $G$ which have degree three in one of parity subgraphs of $\mathcal{J}_0$ is at most \[T(G)\leq |V|-|C_0|\leq |V|-\frac{2}{5} |E|=|V|-\frac{3}{5}|V|=\frac{2}{5}|V|.\]
Here, we used Remark \ref{prop:ParitySubgraphCubic}. Thus, $T(G)\leq \frac{2}{5}|V(G)|$. The proof is complete.
\end{proof}

One may wonder what could be the best upper bound for $T(G)$. The following conjecture  tries to answer this question.

\begin{conjecture}
\label{conj:T(G)upperbound} Let $G$ be a bridgeless cubic graph. Then \[T(G)\leq \frac{|V(G)|}{10}.\]
\end{conjecture} Note that the upper bound in Conjecture \ref{conj:T(G)upperbound} is going to be tight: indeed, there exist infinitely many cubic graphs reaching such a bound (see for instance \cite{MacSkovSIDMA21} for more details), the smallest of them is the Petersen graph (Figure \ref{fig:Petersen10}), then $|V|=10$ and $T(G)=1$ since $P_{12}$ can be covered with four perfect matchings (Figure \ref{fig:Petersen12}).

\begin{figure}[ht]
\centering

  	\begin{center}
	
	\tikzstyle{every node}=[circle, draw, fill=black!50,
                        inner sep=0pt, minimum width=4pt]
                        
		\begin{tikzpicture}[scale=0.75]%[-,>=stealth',shorten >=1pt,auto,node distance=3cm,
			%thick,main node/.style={circle,fill=blue!20,draw,font=\sffamily\Large\bfseries},scale=0.55]
			% \begin{tikzpicture}
																								
			\node[circle,fill=black,draw] at (0,0) (n1) {};
			%\node at (-5, 4.35) (l1) {$r$ $(2,0)$};
																								
			\node[circle,fill=black,draw] at (-2.5, 0) (n2) {};
			% \node at (-3.3, 3.35) (l2) {$(0,1)$};
																								
			\node[circle,fill=black,draw] at (2.5,0) (n3) {};
			%  \node at (-6.8, 3.35) (l2) {$(0,1)$};
																								
			\node[circle,fill=black,draw] at (-1.5,-0.75) (n4) {};
			%\node at (-5, 4.35) (l1) {$r$ $(2,0)$};
																								
			\node[circle,fill=black,draw] at (1.5, -0.75) (n5) {};
			% \node at (-3.3, 3.35) (l2) {$(0,1)$};
																								
			\node[circle,fill=black,draw] at (-1,-2) (n6) {};
			%  \node at (-6.8, 3.35) (l2) {$(0,1)$};
																								
			\node[circle,fill=black,draw] at (1,-2) (n7) {};
			%\node at (-5, 4.35) (l1) {$r$ $(2,0)$};
																								
			\node[circle,fill=black,draw] at (-2, -3) (n8) {};
			% \node at (-3.3, 3.35) (l2) {$(0,1)$};
																								
			\node[circle,fill=black,draw] at (2,-3) (n9) {};
			%  \node at (-6.8, 3.35) (l2) {$(0,1)$};
			%%% n1-n9 are the original vertices of P-v
			
			%%% these three vertices belong to the triangle
			
			\node[circle,fill=black,draw] at (-1,2) (n01) {};
			
			\node[circle,fill=black,draw] at (0,1) (n02) {};
			
			\node[circle,fill=black,draw] at (1,2) (n03) {};

			\path[every node]
			(n1) edge  (n6)
																								    
			edge  (n7)
			%edge (n10)
			%edge [bend right] (n2)
			%edge [bend left] (n2)
			%edge [bend right] (n3)
																								   	
			(n2) edge (n4)
			edge (n8)
			%edge [bend left] (n3)
			%edge [bend right] (n3)
																								       
			(n3) edge (n5)
			edge (n9)
			%  edge [bend left] (n1)
			% edge [bend right] (n1) ;
			(n4) edge (n5)
			edge (n7)

			(n5) edge (n6)
																								   
			(n6)edge (n8)
																								   
			(n7) edge (n9)

			(n8) edge (n9)
			
			(n01) edge (n02)
			      edge (n03)
			      edge (n2)
			 (n02) edge (n03)
			        edge (n1)
			 (n03) edge (n3)
			;

		\end{tikzpicture}
																
	\end{center}
	\caption{The graph $P_{12}$.}\label{fig:Petersen12}

\end{figure}
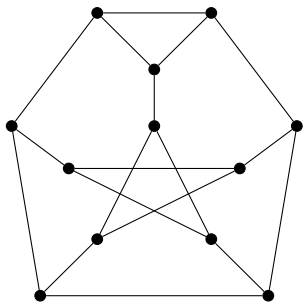

{Also, observe that if Conjecture \ref{conj:T(G)upperbound} is true, then by Proposition \ref{prop:sccT(G)} for any bridgeless cubic graph $G$ we will have

\[scc(G) \leq \frac{4}{3} |E(G)|+T(G)\leq \frac{4}{3} |E(G)|+\frac{|V(G)|}{10}=\frac{7}{5} |E(G)|,\]

which matches $\frac{7}{5}|E|$-bound from Conjecture \ref{conj:75conj} for cubic graphs. Thus, if we can show that Conjecture \ref{conj:T(G)upperbound} follows from Conjecture \ref{conj:52CDC}, then it would mean that Conjecture \ref{conj:52CDC} implies the restriction of Conjecture \ref{conj:75conj} to cubic graphs, something that was not previously known.

Unfortunately, we are not able to derive  Conjecture \ref{conj:T(G)upperbound} directly from 
Conjecture \ref{conj:52CDC}. However, this becomes possible if we start from a stronger version of the $5$-CDC Conjecture.
First of all, note that Conjecture \ref{conj:52CDC} can be reformulated in the following equivalent way.

\begin{conjecture}
\label{conj:5CDC25Bound} Let $G$ be any bridgeless graph. Then $G$ admits a $5$-CDC ${\cal C}=(E_0,...,E_4)$ such that $|E_0|\geq \frac{2}{5}|E(G)|$.
\end{conjecture}

We offer the following seemingly stronger version of the $5$-Cycle Double Cover Cover Conjecture.

\begin{conjecture}
\label{conj:5CDCstrenghtened} Let $G$ be any bridgeless graph. Then $G$ admits a $5$-CDC ${\cal C}=(E_0,...,E_4)$ such that $|E_0|\geq \frac{3}{5}|E(G)|$.
\end{conjecture}

Now, we show that this strengthened version of Conjecture \ref{conj:52CDC}  is enough to derive Conjecture \ref{conj:T(G)upperbound}. }

\begin{theorem}
    \label{thm:35EvenSubgraphImpliesT(G)10} Conjecture \ref{conj:5CDCstrenghtened} implies Conjecture \ref{conj:T(G)upperbound}.
\end{theorem}

\begin{proof}  Let $G$ be a bridgeless cubic graph. Let $(C_0,...,C_4)$ be a $5$-CDC of $G$ with $|C_0|\geq \frac{3|E|}{5}$. Consider the even cover \[\mathcal{C}_0=\{C_0\bigtriangleup C_1,...,C_0\bigtriangleup C_4\}.\] As in \cite{Hou2016}, we have if $e\in C_0$, then $\mathcal{C}_0$ covers $e$ three times, and if $e\notin C_0$, then $\mathcal{C}_0$ covers $e$ two times. Let \[\mathcal{J}_0=\{\overline{C_0\bigtriangleup C_1},...,\overline{C_0\bigtriangleup C_4}\}\] be the set of complements of the subgraphs $C_0\bigtriangleup C_1,...,C_0\bigtriangleup C_4$. Observe that all of them are parity subgraphs. Moreover,  if $e\in C_0$, then $\mathcal{J}_0$ covers $e$ once, and if $e\notin C_0$, then $\mathcal{J}_0$ covers $e$ twice. If an edge $e$ is covered once in $\mathcal{J}_0$, then no parity subgraph of $\mathcal{J}_0$ has degree three on endpoints of $e$. This means that the number of vertices of $G$ which have degree three in one of parity subgraphs of $\mathcal{J}_0$ is at most \[T(G)\leq |V|-|C_0|\leq |V|-\frac{3}{5}|E|=|V|-\frac{9}{10}|V|=\frac{1}{10}|V|.\]
Here we used Remark \ref{prop:ParitySubgraphCubic}. Thus, $T(G)\leq \frac{1}{10}|V|$. The proof is complete.
\end{proof}

One may wonder how realistic is Conjecture \ref{conj:5CDCstrenghtened}: it is not difficult to show that it is a consequence of Petersen Coloring Conjecture.
\begin{theorem}
    \label{thm:P10implies35EvenSubgraph} Petersen Coloring Conjecture implies the restriction of Conjecture \ref{conj:5CDCstrenghtened} to cubic graphs.
\end{theorem}

\begin{proof} Assume that $G$ is a bridgeless cubic graph. Then by Petersen Coloring Conjecture, $G$ admits a Petersen coloring $f$. Since $P_{10}$ has ten vertices, one of its vertices, say $z$, is an image of at most $\frac{|V|}{10}$ vertices under $f$. Let $C$ be a 9-cycle of the Petersen graph that does not pass through $z$. $P_{10}$ admits a 5-CDC such that one of the even subgraphs in it is $C$. Observe that by definition of $z$, at least $\frac{9|V|}{10}$ vertices of $G$ map to vertices of $C$ under $f$. Thus, $f^{-1}(C)$ is an even subgraph of $G$ such that it is part of a 5-CDC of $G$ and its size is at least $\frac{9|V|}{10}=\frac{3}{5}|E|$. The proof is complete.
\end{proof}

\begin{corollary}
    \label{cor:P10T(G)110Bound} Petersen Coloring Conjecture implies Conjecture \ref{conj:T(G)upperbound}.
\end{corollary}

We have derived the restriction of Conjecture \ref{conj:5CDCstrenghtened} to cubic graphs as a consequence of Conjecture \ref{conj:P10conj}. The authors suspect that one should be able to derive Conjecture \ref{conj:5CDCstrenghtened} as a consequence of Conjecture \ref{conj:5CDC25Bound}. In other words, they would like to offer:

\begin{conjecture}
\label{thm:5CDCequivalence} Conjecture \ref{conj:5CDCstrenghtened} is equivalent to Conjecture \ref{conj:5CDC25Bound}.
\end{conjecture}

\section{Expanding Vertices to Obtain a Graph with a Perfect Matching}
\label{CubicResultsSection}

In this section, {we consider cubic graphs that may contain bridges}. For a cubic graph $G$ let $t(G)$ be the size of a smallest subset $U$ of $V(G)$, such that $G_U$ has a perfect matching. Note that $t(G)$ is well-defined and $t(G)\leq |V|$, since if we replace all vertices of $G$ with a triangle then {the set of edges not included in the introduced triangles forms} a perfect matching in $G_V$.

We start by obtaining a Gallai type equality for $t(G)$. If $G$ is an arbitrary cubic graph, then let $\ell(G)$ be the number of edges in a cycle of $G$ having maximum length.

\begin{theorem}\label{thm:GallaiEqualityt(G)}
    Let $G$ be a cubic graph. Then,
    \[|V(G)|=t(G)+\ell(G).\]
\end{theorem}

\begin{proof} We use an argument similar to the one used to prove standard Gallai equalities, see \cite{harary}. Let $C$ be a longest cycle of $G$.  Let us expand the vertices of $G$ in $V(G) \setminus V(C)$ to triangles. Then $C$ together with these new triangles will form a $2$-factor in $G$. The complement of this $2$-factor will be a perfect matching in the expanded graph. Thus:
\[t(G)\leq |V(G)|-|V(C)|=|V(G)|-\ell(G),\]
or
\[\ell(G)+t(G)\leq |V(G)|.\]
For the proof of the converse inequality, let $U\subseteq V(G)$ be a smallest subset of vertices whose expansion to triangles leads to a cubic graph $G_U$ with a perfect matching. Since $G_U$ is cubic, this is equivalent to $G_U$ having a $2$-factor. Since $U$ is minimum, hence $|U|=t(G)$ we have these new triangles will be a part of any $2$-factor of the cubic graph $G_U$. Let $C$ be a cycle of $G$ resulting from the $2$-factor by removing the new triangles. Then it contains
\[|V(G)|-t(G)\]
vertices and that many edges. Hence,
%\[\ell(G)\geq |V|-t(G),\]
%or
\[\ell(G)+t(G)\geq |V(G)|.\]
    The proof is complete.
\end{proof}

Next, we prove a lemma that will be used later in order to obtain the main result of this section.

\begin{lemma}
    \label{lem:BridgelessMatchingSubdivide} Let $G$ be a bridgeless cubic graph and let $E_0\subseteq E(G)$. Consider the cubic graph $H$ obtained from $G$ by subdividing every edge $e\in E_0$ and adding a copy of $W$ to it. Then,
    \[t(H) \leq \min_{M} |E_0 \cap M|\]
    where $M$ is a perfect matching of $G$.  
\end{lemma}

\begin{proof} 
Let $M$ be a perfect matching of $G$ (see Theorem \ref{thm:Petersen}) such that $|E_0 \cap M|$ is minimum. 
Let $U$ be the subset of $V(H)$ consisting of all roots of a copy of $W$ which subdivide an edge of $G$ in $E_0 \cap M$. It suffices to prove the existence of a perfect matching in $H_U$. We construct a perfect matching $N$ of $H_U$ arising from the perfect matching $M$ of $G$ as follows:
if an edge $e\notin E_0$, then the corresponding edge in $H_U$ belongs to $N$ only if $e$ belongs to $M$.
If an edge $e=xy\in E_0$, then the perfect matching $N$ of $H_U$ contains one of the two parallel edges of the copy of $W$ corresponding to $e$, and the bridge joining such a copy of $W$ to the rest of the graph $H_U$: moreover, if $xy \in M$ then $N$ contains also the two edges incident with $x$ and $y$ and with the triangle obtained by the expansion of the root of $W$ in $H$. The assertion follows.
\end{proof}

Next lemma is a direct consequence of previous one and the observation that for every $h, g\in E_0$, $G-h-g$ has a perfect matching ((a) of Theorem \ref{thm:Petersen}).

\begin{lemma}\label{lem:E2}
    Let $G$ be a bridgeless cubic graph and let $E_0\subseteq E(G)$ consisting of at least two edges. Consider the cubic graph $H$ obtained from $G$ by subdividing every edge $e\in E_0$ and adding a copy of $W$ to it. Then, 
    \[ t(H) \leq |E_0|-2.\]
\end{lemma}

In our next theorem, we obtain asymptotically tight upper bounds for $t(G)$ in the classes of cubic graphs and simple cubic graphs.

\begin{theorem}
    \label{thm:t(G)upperboundTight} Let $G$ be a cubic graph. Then, the following holds:
    \begin{enumerate}
        \item [(i)] $t(G)< \frac{|V|}{4}$;
        
        \item [(ii)] if $G$ is simple, then  $t(G)< \frac{|V|}{6}$;

    \end{enumerate}

Moreover, previous bounds are asymptotically tight.

\end{theorem}

\begin{proof} For the proof of statement $(i)$, assume that $G$ is a counterexample of minimum order. Observe that $t(G)\geq \frac{1}{4}|V(G)|$ implies $\frac{\ell(G)}{|V(G)|} \leq \frac34$ by Theorem \ref{thm:GallaiEqualityt(G)}. Clearly, $G$ is connected. By Theorem \ref{thm:Petersen}, we can assume that $G$ contains at least three bridges. We proceed by proving a series of intermediate statements that ultimately yield a contradiction, disproving the existence of such a counterexample.

\begin{claim}
    \label{claim:TreeEndblocks3} Every end-block of $G$ is isomorphic to $W$.
\end{claim}

\begin{proof} Suppose some end-block $B$ has $h\geq 5$ vertices. Define a smaller cubic graph $G'$ obtained from $G$ by removing $B$ from $G$ and attaching a copy of $W$ (Figure \ref{fig:claim1}). 

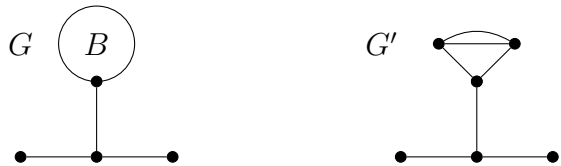
\begin{figure}[ht]
  
  \begin{center}

		\begin{tikzpicture}%[-,>=stealth',shorten >=1pt,auto,node distance=3cm,
		%	thick,main node/.style={circle,fill=blue!20,draw,font=\sffamily\Large\bfseries}]
			% \begin{tikzpicture}

            %\node at (-10, -2.35) {$e$};

            %\node at (-10.5, -1.65) {$u$};
            %\node at (-9.5, -1.65) {$v$};

            %\node at (-6.5, -1.65) {$u$};
            %\node at (-4.5, -1.65) {$v$};
            %\node at (-5.5, -2.35) {$w_e$};

            %\draw (18:1cm) circle (2pt);
            \draw (-10.5,-0.5) circle (0.5cm);
            \node at (-10.5, -0.5) {$B$};
            \node at (-11.5, -0.5) {$G$};
             \node at (-6.75, -0.5) {$G'$};

            \tikzstyle{every node}=[circle, draw, fill=black!50,
                        inner sep=0pt, minimum width=4pt]
																							
			\node[circle,fill=black,draw] at (-5.5,-1) (n1) {};
			%\draw[black,fill=black,thick] (-5.5,-1) circle [radius=0.1cm] ;
			%\node at (-5, 4.35) (l1) {$r$ $(2,0)$};
			
			%\draw[fill=black,thick] (-5.5,-1) circle (2pt);
			%\node at (-5.5, -1) (n1){};
			
			%\draw[fill=black] (18:1cm) circle (2pt);
			
			%\draw[fill=black] (240:1cm) circle (2pt);
			
			%\draw[fill=black] ((-5.5,-1):1cm) circle (2pt);

                \node[circle,fill=black,draw] at (-5.5,-1) (n1) {};
                \node[circle,fill=black,draw] at (-10.5,-1) (n11) {};
																								
			\node[circle,fill=black,draw] at (-6, -0.5) (n2) {};
			% \node at (-3.3, 3.35) (l2) {$(0,1)$};
																								
			\node[circle,fill=black,draw] at (-5,-0.5) (n3) {};
			%  \node at (-6.8, 3.35) (l2) {$(0,1)$};

                \node[circle,fill=black,draw] at (-5.5,-2) (n4) {};
                \node[circle,fill=black,draw] at (-10.5,-2) (n44) {};

                \node[circle,fill=black,draw] at (-6.5,-2) (n5) {};
                \node[circle,fill=black,draw] at (-11.5,-2) (n55) {};
                
                \node[circle,fill=black,draw] at (-4.5,-2) (n6) {};
                \node[circle,fill=black,draw] at (-9.5,-2) (n66) {};

                 %\node[circle,fill=black,draw] at (-10.5,-2) (n7) {};
                %\node[circle,fill=black,draw] at (-9.5,-2) (n8) {};

			\path[every node]
			(n1) edge  (n2)
                edge (n4)

			edge  (n3)
			
			%edge [bend right] (n2)
			%edge [bend left] (n2)
			%edge [bend right] (n3)
																								   	
			(n2) edge (n3)
			edge [bend left] (n3)
			%edge [bend left] (n3)
			%edge [bend right] (n3)
                (n4) edge (n5)
                (n4) edge (n6)

                %(n7) edge (n8)
                (n11) edge (n44)  
                (n44) edge (n55)
                (n44) edge (n66)

			;
		\end{tikzpicture}
																
	\end{center}
								
	\caption{Obtaining the cubic graph $G'$ from $G$.}
	\label{fig:claim1}

\end{figure}

We have:
\[|V(G')|=|V(G)|-(h-3),\]
and
\[\ell(G')=\ell(G)-(h-3).\]
The latter follows from observation that the end-block has a 2-factor (see Corollary \ref{cor:Endblock2factor}), hence every maximum even subgraph covers all these $h$ vertices. Since $h\geq 5$, we have $|V(G')|<|V(G)|$. Thus,  by minimality of $G$, we have
\[\frac{\ell(G)}{|V(G)|}=\frac{\ell(G')+(h-3)}{|V(G')|+(h-3)}\geq \frac{\ell(G')}{|V(G')|}> \frac{3}{4}\]
contradicting that $G$ is a counterexample. Here we also used the fact that $\ell(G')\leq |V(G')|$.
The proof of Claim \ref{claim:TreeEndblocks3} is complete.
\end{proof}

\begin{claim}
    \label{claim:AdjacentEndBlocks} There is no vertex $w$ of $G$, such that two end-blocks of $G$ are joined to $w$ with two bridges. In other words, different end-blocks have different roots.
\end{claim}

\begin{proof} On the opposite assumption, assume that $w$ is incident to vertices $x$ and $y$ such that $x$ and $y$ lie in end-blocks (Figure \ref{fig:claim2}). By Claim \ref{claim:TreeEndblocks3}, these end-blocks have three vertices. Consider the cubic graph $H$ obtained from $G$ by removing these two end-blocks and attaching a copy of $W$ to $w$ (Figure \ref{fig:claim2}). %Note that $w$ becomes the degree-two vertex of the new end-block of $H$. 

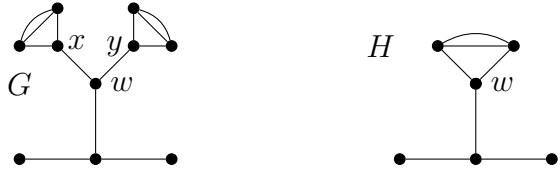
\begin{figure}[ht]
  
  \begin{center}

		\begin{tikzpicture}%[-,>=stealth',shorten >=1pt,auto,node distance=3cm,
		%	thick,main node/.style={circle,fill=blue!20,draw,font=\sffamily\Large\bfseries}]
			% \begin{tikzpicture}

            %\node at (-10, -2.35) {$e$};

            %\node at (-10.5, -1.65) {$u$};
            %\node at (-9.5, -1.65) {$v$};

            %\node at (-6.5, -1.65) {$u$};
            %\node at (-4.5, -1.65) {$v$};
            %\node at (-5.5, -2.35) {$w_e$};

            %\draw (18:1cm) circle (2pt);
            %\draw (-10.5,-0.5) circle (0.5cm);
            %\node at (-10.5, -0.5) {$B$};
            \node at (-11.5, -1) {$G$};
             \node at (-6.75, -0.5) {$H$};

             \node at (-10.15, -1) {$w$};
              \node at (-10.75, -0.45) {$x$};
              \node at (-10.25, -0.5) {$y$};
              
             \node at (-5.15, -1) {$w$};

            \tikzstyle{every node}=[circle, draw, fill=black!50,
                        inner sep=0pt, minimum width=4pt]
																							
			\node[circle,fill=black,draw] at (-5.5,-1) (n1) {};
			%\draw[black,fill=black,thick] (-5.5,-1) circle [radius=0.1cm] ;
			%\node at (-5, 4.35) (l1) {$r$ $(2,0)$};
			
			%\draw[fill=black,thick] (-5.5,-1) circle (2pt);
			%\node at (-5.5, -1) (n1){};
			
			%\draw[fill=black] (18:1cm) circle (2pt);
			
			%\draw[fill=black] (240:1cm) circle (2pt);
			
			%\draw[fill=black] ((-5.5,-1):1cm) circle (2pt);

                \node[circle,fill=black,draw] at (-5.5,-1) (n1) {};
                \node[circle,fill=black,draw] at (-10.5,-1) (n11) {};
                \node[circle,fill=black,draw] at (-10,-0.5) (n111) {};
                \node[circle,fill=black,draw] at (-10,0) (n1111) {};
                \node[circle,fill=black,draw] at (-9.5,-0.5) (n1112) {};
                
                \node[circle,fill=black,draw] at (-11,-0.5) (n112) {};
                \node[circle,fill=black,draw] at (-11.5,-0.5) (n1121) {};
                \node[circle,fill=black,draw] at (-11,0) (n1122) {};
                
			\node[circle,fill=black,draw] at (-6, -0.5) (n2) {};
			% \node at (-3.3, 3.35) (l2) {$(0,1)$};
																								
			\node[circle,fill=black,draw] at (-5,-0.5) (n3) {};
			%  \node at (-6.8, 3.35) (l2) {$(0,1)$};

                \node[circle,fill=black,draw] at (-5.5,-2) (n4) {};
                \node[circle,fill=black,draw] at (-10.5,-2) (n44) {};

                \node[circle,fill=black,draw] at (-6.5,-2) (n5) {};
                \node[circle,fill=black,draw] at (-11.5,-2) (n55) {};
                
                \node[circle,fill=black,draw] at (-4.5,-2) (n6) {};
                \node[circle,fill=black,draw] at (-9.5,-2) (n66) {};

                 %\node[circle,fill=black,draw] at (-10.5,-2) (n7) {};
                %\node[circle,fill=black,draw] at (-9.5,-2) (n8) {};

			\path[every node]
			(n1) edge  (n2)
                edge (n4)

			edge  (n3)
			
			%edge [bend right] (n2)
			%edge [bend left] (n2)
			%edge [bend right] (n3)
																								   	
			(n2) edge (n3)
			edge [bend left] (n3)
			%edge [bend left] (n3)
			%edge [bend right] (n3)
                (n4) edge (n5)
                (n4) edge (n6)

                %(n7) edge (n8)
                (n11) edge (n44)  
                (n44) edge (n55)
                (n44) edge (n66)

                (n11) edge (n111)
                (n11) edge (n112)

                (n111) edge (n1111) 
                (n111) edge (n1112) 
                
                (n1111) edge (n1112)
                edge [bend left] (n1112)

                (n1121) edge (n1122)
                edge [bend left] (n1122)

                (n112) edge (n1121)
                edge (n1122)

			;
		\end{tikzpicture}
																
	\end{center}
								
	\caption{Obtaining the cubic graph $H$ from $G$.}
	\label{fig:claim2}

\end{figure}

We have
\[|V(G)|=|V(H)|+4,\]
and
\[\ell(G)=\ell(H)+3.\]
Since $|V(H)|<|V(G)|$, by minimality of $G$, we have
\[\frac{\ell(H)}{|V(H)|}> \frac{3}{4}.\]
Hence, 
\[\frac{\ell(G)}{|V(G)|}=\frac{\ell(H)+3}{|V(H)|+4}> \frac{3}{4},\]
which contradicts our choice of $G$ as a counterexample. The proof of Claim \ref{claim:AdjacentEndBlocks} is complete.
\end{proof}

Our next claim states that roots of end-blocks form an independent set.

\begin{claim}
    \label{claim:AdjacentRootsEndBlocks} There are no two roots $w_1$ and $w_2$ of end-blocks in $G$, such that $w_1w_2\in E(G)$.
\end{claim}

\begin{proof} Suppose two roots $w_1$ and $w_2$ are adjacent in $G$ (Figure \ref{fig:claim3}). 
First of all, let us observe that $w_1w_2$ cannot be a double edge since $G$ has more than two bridges. Moreover, by previous claim there are no two end-blocks of $G$ joined in $w_2$. Then, we denote by $w_3$ the neighbour of $w_2$ distinct from $w_1$ and such that $w_2w_3$ is not the bridge in an end-block.  

Consider the cubic graph $H$ obtained from $G$ by removing the end-block having $w_2$ as a root and adding a new edge $w_1w_3$ (Figure \ref{fig:claim3}). Note that $w_1w_3$ could be a double edge in $H$.

\begin{figure}[ht]
  
  \begin{center}

		\begin{tikzpicture}%[-,>=stealth',shorten >=1pt,auto,node distance=3cm,
		%	thick,main node/.style={circle,fill=blue!20,draw,font=\sffamily\Large\bfseries}]
			% \begin{tikzpicture}

            %\node at (-10, -2.35) {$e$};

            %\node at (-10.5, -1.65) {$u$};
            %\node at (-9.5, -1.65) {$v$};

            %\node at (-6.5, -1.65) {$u$};
            %\node at (-4.5, -1.65) {$v$};
            %\node at (-5.5, -2.35) {$w_e$};

            %\draw (18:1cm) circle (2pt);
           % \draw (-10.5,-0.5) circle (0.5cm);
            %\node at (-10.5, -0.5) {$B$};
            %\node at (-11.5, -0.5) {$G$};
             %\node at (-6.75, -0.5) {$G'$};

             \node at (-9, -2) {$G$}; \node at (-0.5, -2) {$H$};

             \node at (-8, -2) {$w_1$}; \node at (-2.35, -1.9) {$w_1$};
             \node at (-5, -2) {$w_2$};
             \node at (-4,-3) {$w_3$}; \node at (0.35,-3) {$w_3$};
             \draw (-7.5,-2) -- (-9,-3) ;
             \draw (-5.5,-2) -- (-4.5,-3) ;
             \draw (-2,-2)--(-3.5,-3);

             \draw[dashed] (-4.5,-3) -- (-7.5,-2);

            \tikzstyle{every node}=[circle, draw, fill=black!50,
                        inner sep=0pt, minimum width=4pt]
																							
			%\node[circle,fill=black,draw] at (-5.5,-1) (n1) {};
			%\draw[black,fill=black,thick] (-5.5,-1) circle [radius=0.1cm] ;
			%\node at (-5, 4.35) (l1) {$r$ $(2,0)$};
			
			%\draw[fill=black,thick] (-5.5,-1) circle (2pt);
			%\node at (-5.5, -1) (n1){};
			
			%\draw[fill=black] (18:1cm) circle (2pt);
			
			%\draw[fill=black] (240:1cm) circle (2pt);
			
			%\draw[fill=black] ((-5.5,-1):1cm) circle (2pt);

                \node[circle,fill=black,draw] at (-5.5,-1) (n1) {};
                \node[circle,fill=black,draw] at (-7.5,-1) (n11) {};
                \node[circle,fill=black,draw] at (-4.5,-3) (n111) {};
            %    \node[circle,fill=black,draw] at (-10.5,-1) (n11) {};

            \node[circle,fill=black,draw] at (-2,-2) (n1H) {};
            \node[circle,fill=black,draw] at (0,-3) (n3H) {};
            \node[circle,fill=black,draw] at (-2,-1) (w1H) {};
            \node[circle,fill=black,draw] at (-2.5,-0.5) (w11H) {};
            \node[circle,fill=black,draw] at (-1.5,-0.5) (w12H) {};

			\node[circle,fill=black,draw] at (-6, -0.5) (n2) {};
                \node[circle,fill=black,draw] at (-8, -0.5) (n22) {};
			% \node at (-3.3, 3.35) (l2) {$(0,1)$};
																								
			\node[circle,fill=black,draw] at (-5,-0.5) (n3) {};
                \node[circle,fill=black,draw] at (-7,-0.5) (n33) {};
			%  \node at (-6.8, 3.35) (l2) {$(0,1)$};

                \node[circle,fill=black,draw] at (-5.5,-2) (n4) {};
                \node[circle,fill=black,draw] at (-7.5,-2) (n44) {};
             %   \node[circle,fill=black,draw] at (-10.5,-2) (n44) {};

               % \node[circle,fill=black,draw] at (-6.5,-2) (n5) {};
              %  \node[circle,fill=black,draw] at (-11.5,-2) (n55) {};
                
               % \node[circle,fill=black,draw] at (-4.5,-2) (n6) {};
               % \node[circle,fill=black,draw] at (-9.5,-2) (n66) {};

                 %\node[circle,fill=black,draw] at (-10.5,-2) (n7) {};
                %\node[circle,fill=black,draw] at (-9.5,-2) (n8) {};

			\path[every node]
                 (w1H) edge (w11H)
                  (w1H) edge (w12H)
                (w11H) edge (w12H)
                edge [bend left] (w12H)
                (w1H) edge (n1H)
                (n1H) edge (n3H)
			(n1) edge  (n2)
                edge (n4)

			edge  (n3)
			
			%edge [bend right] (n2)
			%edge [bend left] (n2)
			%edge [bend right] (n3)
																								   	
			(n2) edge (n3)
			edge [bend left] (n3)
			%edge [bend left] (n3)
			%edge [bend right] (n3)
                %(n4) edge (n5)
                %(n4) edge (n6)
                (n4) edge (n44)

                %(n7) edge (n8)
                
				(n44) edge (n11)

                (n11) edge (n22)
                (n11) edge (n33)
                (n22) edge (n33)
                edge [bend left] (n33)

			;
		\end{tikzpicture}
																
	\end{center}
								
	\caption{Obtaining the cubic graph $H$ from $G$.}
	\label{fig:claim3}

\end{figure}
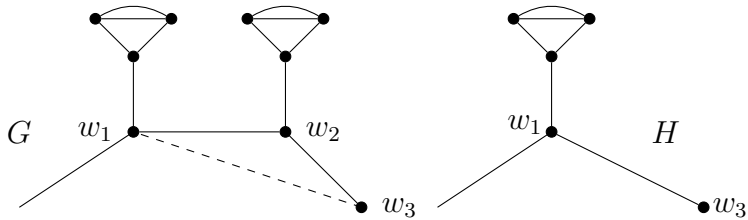

We have:
\[|V(G)|=|V(H)|+4,\]
and
\[\ell(G)\geq \ell(H)+3.\]
Since $|V(H)|<|V(G)|$, we have
\[\frac{\ell(H)}{|V(H)|}> \frac{3}{4}\]
by minimality of $G$. Hence
\[\frac{\ell(G)}{|V(G)|}\geq \frac{\ell(H)+3}{|V(H)|+4} > \frac{3}{4},\]
which contradicts our choice of $G$ as a counterexample. The proof of Claim \ref{claim:AdjacentRootsEndBlocks} is complete.
\end{proof}

%\begin{claim}
%    \label{claim:EndBlocksRootsNeighbors} ({\bf It is quite plausible that we do not need this claim. Check later!!!}) There are no two roots $w_1$ and $w_2$ of end-blocks, such that there is a vertex $z$ with $zw_1, zw_2\in E(G)$.
%\end{claim} 
%
%\begin{proof} Suppose that a vertex $z$ is adjacent to two roots $w_1$ and $w_2$ of end-blocks. Consider a graph $H$ obtained from $G$ by contracting $z$ and these two end-blocks and $w_1$, $w_2$ to a single vertex $u$. Note that $H$ is a cubic graph. Moreover,
%\[|V(G)|=|V(H)|-1+9=|V(H)|+8,\]
%and
%\[\ell(G)\geq \ell(H)+6.\]
%Since $|V(H)|<|V(G)|$, by minimality of $G$, we have
%\[\frac{\ell(H)}{|V(H)|}\geq \frac{3}{4}.\]
%Thus,
%\[\frac{\ell(G)}{|V(G)|}\geq \frac{\ell(H)+6}{|V(H)|+8}\geq \frac{3}{4},\]
%contradicting our choice of $G$ as a counter-example. The proof of the claim is complete.
%\end{proof}

The proved claims allow us to view our counterexample $G$ as one obtained from a cubic graph $G_0$ by taking a subset $E_0\subseteq E(G_0)$, subdividing edges of $E_0$ once and attaching a copy of $W$ to them. %(see Lemma \ref{lem:BridgelessMatchingSubdivide})

Let us say that a bridge $e$ of $G$ is trivial if \[\min\{|V(G_1)|, |V(G_2)|\}=3,\] where $G_1$ and $G_2$ denote the components of $G-e$.

\begin{claim}
    \label{claim:NonTrivialBridge} $G$ contains a non-trivial bridge.
\end{claim}

\begin{proof} Assume that all bridges in $G$ are trivial. Then the graph $G_0$ discussed above is a bridgeless cubic graph. By Lemma \ref{lem:BridgelessMatchingSubdivide},
\[t(G)\leq |E_0|\leq \frac{|V(G)|-|V(G_0)|}{4} < \frac{|V(G)|}{4},\]
since for every edge $e\in E_0$ we have four vertices in $G$. This contradicts the assumption that $G$ is a counterexample. The proof of Claim \ref{claim:NonTrivialBridge} is complete.
\end{proof}

We are ready to complete the proof of statement $(i)$. Let $f$ be a non-trivial bridge in $G$. Let $G_1$ and $G_2$ be the components of $G-f$. Consider two cubic graphs $H_1$ and $H_2$ obtained from $G_1$ and $G_2$, respectively, by attaching copies of $W$ to the end-vertices of $f$ (Figure \ref{fig:proofend}). 

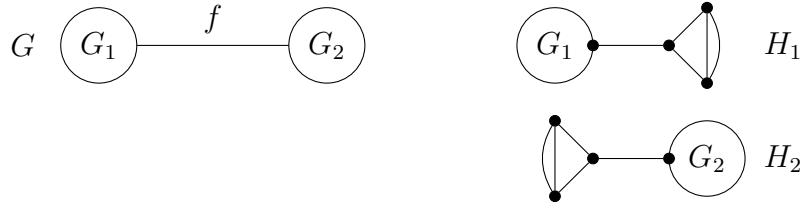
\begin{figure}[ht]
  
  \begin{center}

		\begin{tikzpicture}%[-,>=stealth',shorten >=1pt,auto,node distance=3cm,
		%	thick,main node/.style={circle,fill=blue!20,draw,font=\sffamily\Large\bfseries}]
			% \begin{tikzpicture}

            %\node at (-10, -2.35) {$e$};

            %\node at (-10.5, -1.65) {$u$};
            %\node at (-9.5, -1.65) {$v$};

            %\node at (-6.5, -1.65) {$u$};
            %\node at (-4.5, -1.65) {$v$};
            %\node at (-5.5, -2.35) {$w_e$};

            %\draw (18:1cm) circle (2pt);
            \draw (-10.5,-0.5) circle (0.5cm);
            \draw (-7.5,-0.5) circle (0.5cm);
            \node at (-10.5, -0.5) {$G_1$}; 
            \node at (-4.5, -0.5) {$G_1$}; 
            \draw (-4.5,-0.5) circle (0.5cm);

            \node at (-1.5, -0.5) {$H_1$};
            \node at (-1.5, -2) {$H_2$};
             \draw (-2.5,-2) circle (0.5cm);
             \node at (-2.5, -2) {$G_2$}; 
            
            \node at (-7.5, -0.5) {$G_2$};

            \node at (-11.5, -0.5) {$G$};
             \node at (-9, -0.15) {$f$};
            % \node at (-6.75, -0.5) {$G'$};
            \draw (-10,-0.5)--(-8,-0.5);

            \tikzstyle{every node}=[circle, draw, fill=black!50,
                        inner sep=0pt, minimum width=4pt]
																							
			\node[circle,fill=black,draw] at (-4,-0.5) (n1) {};
                 \node[circle,fill=black,draw] at (-3,-0.5) (n2) {};
                 \node[circle,fill=black,draw] at (-2.5,0) (n21) {};
                 \node[circle,fill=black,draw] at (-2.5,-1) (n22) {};

                 \node[circle,fill=black,draw] at (-3,-2) (n1a) {};
                 \node[circle,fill=black,draw] at (-4,-2) (n2a) {};
                  \node[circle,fill=black,draw] at (-4.5,-1.5) (n21a) {};
                  \node[circle,fill=black,draw] at (-4.5,-2.5) (n22a) {};
			%\draw[black,fill=black,thick] (-5.5,-1) circle [radius=0.1cm] ;
			%\node at (-5, 4.35) (l1) {$r$ $(2,0)$};
			
			%\draw[fill=black,thick] (-5.5,-1) circle (2pt);
			%\node at (-5.5, -1) (n1){};
			
			%\draw[fill=black] (18:1cm) circle (2pt);
			
			%\draw[fill=black] (240:1cm) circle (2pt);
			
			%\draw[fill=black] ((-5.5,-1):1cm) circle (2pt);

               % \node[circle,fill=black,draw] at (-5.5,-1) (n1) {};
              %  \node[circle,fill=black,draw] at (-10.5,-1) (n11) {};
																								
			%\node[circle,fill=black,draw] at (-6, -0.5) (n2) {};
			% \node at (-3.3, 3.35) (l2) {$(0,1)$};
																								
			%\node[circle,fill=black,draw] at (-5,-0.5) (n3) {};
			%  \node at (-6.8, 3.35) (l2) {$(0,1)$};

                %\node[circle,fill=black,draw] at (-5.5,-2) (n4) {};
               % \node[circle,fill=black,draw] at (-10.5,-2) (n44) {};

                %\node[circle,fill=black,draw] at (-6.5,-2) (n5) {};
               % \node[circle,fill=black,draw] at (-11.5,-2) (n55) {};
                
                %\node[circle,fill=black,draw] at (-4.5,-2) (n6) {};
               % \node[circle,fill=black,draw] at (-9.5,-2) (n66) {};

                 %\node[circle,fill=black,draw] at (-10.5,-2) (n7) {};
                %\node[circle,fill=black,draw] at (-9.5,-2) (n8) {};

			\path[every node]
			(n1) edge  (n2)
                (n2) edge (n21)
                (n2) edge (n22)
                (n21) edge (n22)
                edge [bend left] (n22)

                (n1a) edge (n2a)

                (n2a) edge (n21a)
                (n2a) edge (n22a)
                (n21a) edge (n22a)
                edge [bend right] (n22a)

			%edge  (n3)
			
			%edge [bend right] (n2)
			%edge [bend left] (n2)
			%edge [bend right] (n3)
																								   	
			%(n2) edge (n3)
			%edge [bend left] (n3)
			%edge [bend left] (n3)
			%edge [bend right] (n3)
                %(n4) edge (n5)
                %(n4) edge (n6)

			;
		\end{tikzpicture}
																
	\end{center}
								
	\caption{Obtaining the cubic graphs $H_1$ and $H_2$ from $G$.}
	\label{fig:proofend}

\end{figure}

Since $f$ is non-trivial, we have $|V(H_1)|, |V(H_2)|<|V(G)|$. Note that we can always choose the non-trivial bridge $f$ such that

\begin{enumerate}
    \item [(a)] $H_1$ contains at least two end-blocks,

    \item [(b)] all bridges in $H_1$ are trivial.
\end{enumerate}

We have:
\[|V(H_1)|=|V(G_1)|+3,|V(H_2)|=|V(G_2)|+3,\]
and since $f$ is a bridge
\[\ell(H_1)=\ell(G_1)+3,\ell(H_2)=\ell(G_2)+3.\]
Since $|V(H_2)|<|V(G)|$, we have
\[\ell(G_2)+3=\ell(H_2)> \frac{3}{4}|V(H_2)|=\frac{3}{4} \left(|V(G_2)|+3\right)\]
by minimality of $G$. This is equivalent to
\begin{equation}
    \label{eq:G2} \ell(G_2)\geq \frac{3}{4} |V(G_2)|-\frac{3}{4}.
\end{equation} On the other hand, since in $H_1$ all bridges are trivial, by Lemma \ref{lem:E2} and minimality of $G$
\[t(H_1)\leq |E_0(H_1)|-2\leq \frac{|V(H_1)|}{4}-2<\frac{|V(H_1)|}{4}-\frac{3}{2}.\]
By Theorem \ref{thm:GallaiEqualityt(G)}, the latter means
\[\ell(H_1)>\frac{3|V(H_1)|}{4}+\frac{3}{2}.\]
Since $\ell(H_1)=\ell(G_1)+3$ and $|V(H_1)|=|V(G_1)|+3$, the last inequality is equivalent to
\begin{equation}
    \label{eq:G1} \ell(G_1)> \frac{3}{4}|V(G_1)|+\frac{3}{4}.
\end{equation} (\ref{eq:G2}) and (\ref{eq:G1}) together imply
\[\ell(G)=\ell(G_1)+\ell(G_2)> \frac{3}{4}|V(G_1)|+\frac{3}{4}+\frac{3}{4}|V(G_2)|-\frac{3}{4}=\frac{3}{4}|V(G)|\]
which contradicts our assumption that $G$ is a counterexample. In the last equation, we used equalities
\[|V(G)|=|V(G_1)|+|V(G_2)|,\]
and
\[\ell(G)=\ell(G_1)+\ell(G_2).\]
The latter follows from our choice of $f$ as a bridge. Proof of $(i)$ is complete.

\medskip

Now, we prove that the bound in $(i)$ is asymptotically tight. Let $T$ be a tree on $n$ vertices in which every vertex is of degree three or one. Let $k_1$ be the number of vertices with degree one, and let $k_3$ be the number of vertices of degree three. Attach a copy of $W$ to every vertex of degree one in $T$ so that we get a cubic graph $G$ (if $T=K_{1,3}$ is the claw, then $G=S_{10}$, see Figure \ref{fig:SylvGraph}).

By Proposition \ref{prop:31degreeTrees},
\[|V(G)|=k_3+3k_1=3k_1+(k_1-2)=4k_1-2=4\left(\frac{n}{2}+1\right)-2=2n+2.\]
Since $T$ is a tree, all its edges are bridges in $G$. Thus, these edges cannot lie on a cycle, hence on even subgraphs of $G$. Thus, in order to get an even subgraph, we can take just triangles in copies of $W$ attached to degree one vertices. Thus:
\[\ell(G)=3k_1.\]
Hence, by Theorem \ref{thm:GallaiEqualityt(G)}
\[t(G)=k_1-2=\left(\frac{n}{2}+1\right)-2=\frac{n}{2}-1.\]
Thus:
\[\lim_{n\rightarrow +\infty}\frac{t(G)}{|V(G)|}=\lim_{n\rightarrow +\infty}\frac{\frac{n}{2}-1}{2n+2}=\lim_{n\rightarrow +\infty}\frac{n-2}{4n+4}=\frac{1}{4}.\]

\medskip

The proof of point $ii)$ proceeds in exactly the same way as in point $i)$, replacing $W$ with $W'$. The only difference lies in Claim \ref{claim:AdjacentRootsEndBlocksSimple}, which corresponds to Claim \ref{claim:AdjacentRootsEndBlocks}, where an additional argument is required to ensure that no parallel edges are created. For this reason, we defer the detailed proof of this point to the Appendix.

\medskip

In order to prove that the bound in $(ii)$ is the best possible, again, we start with a tree $T$ in which all degrees are either one or three. Then we attach to every vertex of degree one a copy of $W'$. For example, if $T=K_{1,3}$ is the claw, then we get the cubic graph from Figure \ref{fig:S16}.

\begin{figure}[ht]
  \begin{center}
	
	\tikzstyle{every node}=[circle, draw, fill=black!50,
                        inner sep=0pt, minimum width=4pt]
	
		\begin{tikzpicture}%[-,>=stealth',shorten >=1pt,auto,node distance=3cm,
		%	thick,main node/.style={circle,fill=blue!20,draw,font=\sffamily\Large\bfseries}]
			% \begin{tikzpicture}
																							
			\node[circle,fill=black,draw] at (-5.5,-1) (n1) {};
			
			\node[circle,fill=black,draw] at (-6, -0.5) (n2) {};
			% \node at (-3.3, 3.35) (l2) {$(0,1)$};
			\node[circle,fill=black,draw] at (-6, 0.5) (n22) {};
																								
			\node[circle,fill=black,draw] at (-5,-0.5) (n3) {};
			%  \node at (-6.8, 3.35) (l2) {$(0,1)$};
			\node[circle,fill=black,draw] at (-5,0.5) (n33) {};
																								
			\node[circle,fill=black,draw] at (-3.5,-1) (n4) {};
			%\node at (-5, 4.35) (l1) {$r$ $(2,0)$};

			\node[circle,fill=black,draw] at (-4, -0.5) (n5) {};
			% \node at (-3.3, 3.35) (l2) {$(0,1)$};
			\node[circle,fill=black,draw] at (-4, 0.5) (n55) {};
																								
			\node[circle,fill=black,draw] at (-3,-0.5) (n6) {};
			%  \node at (-6.8, 3.35) (l2) {$(0,1)$};
			\node[circle,fill=black,draw] at (-3, 0.5) (n66) {};
																								
			\node[circle,fill=black,draw] at (-1.5,-1) (n7) {};
			%\node at (-5, 4.35) (l1) {$r$ $(2,0)$};
																								
			\node[circle,fill=black,draw] at (-2, -0.5) (n8) {};
			% \node at (-3.3, 3.35) (l2) {$(0,1)$};
			\node[circle,fill=black,draw] at (-2, 0.5) (n88) {};
																								
			\node[circle,fill=black,draw] at (-1,-0.5) (n9) {};
			%  \node at (-6.8, 3.35) (l2) {$(0,1)$};
			\node[circle,fill=black,draw] at (-1, 0.5) (n99) {};
																								
			\node[circle,fill=black,draw] at (-3.5,-2) (n10) {};

			\path[every node]
			(n1) edge  (n2)

			edge  (n3)
			edge (n10)
			%edge [bend right] (n2)
			%edge [bend left] (n2)
			%edge [bend right] (n3)
																								   	
			(n2) edge (n22)
			    edge (n33)
		%	edge [bend left] (n3)
			%edge [bend left] (n3)
			%edge [bend right] (n3)
																								       
			(n3) edge (n22)
			    edge (n33)
			%  edge [bend left] (n1)
			% edge [bend right] (n1) ;
			(n4) edge (n5)
			edge (n6)
			edge (n10)
			
			(n22) edge (n33)
																								    
			(n5) edge (n55)
			    edge (n66)
		%	edge [bend left] (n6)
			(n6) edge (n55)
			    edge (n66)
			    
			 (n55) edge (n66)
																								   
			(n7) edge (n8)
			edge (n9)
			edge (n10)
																								    
			(n8) edge (n88)
			    edge (n99)
		%	edge [bend left] (n9)
		(n9) edge (n88)
		    edge (n99)
		   
		 (n88) edge (n99)
																								  
			;
		\end{tikzpicture}
																
	\end{center}
								
	\caption{The graph $S_{16}$.}
	\label{fig:S16}

\end{figure}
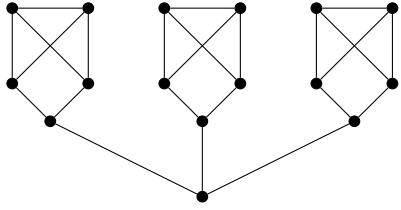

We will have:
\[|V(G)|=k_1-2+5k_1=6k_1-2=6\left(\frac{n}{2}+1\right)-2=3n+4,\]
and
\[\ell(G)=5k_1.\]
Thus, by Theorem \ref{thm:GallaiEqualityt(G)},
\[t(G)=k_1-2=\frac{n}{2}-1,\]
and
\[\lim_{n\rightarrow +\infty}\frac{t(G)}{|V(G)|}=\lim_{n\rightarrow +\infty}\frac{\frac{n}{2}-1}{3n+4}=\lim_{n\rightarrow +\infty}\frac{n-2}{6n+8}=\frac{1}{6}.\]
\end{proof}

\section{Computational complexity of introduced parameters}\label{sec:complexity}

In this section, we discuss the computational complexity of computing our parameters when on the input one is given a cubic graph. We start with the discussion of $t(G)$. 

If $J$ is a parity subgraph of a cubic graph $G$, then let $V_1(J)$ and $V_3(J)$ be the sets of vertices of degree one and three in $J$, respectively. Note that by Remark \ref{prop:ParitySubgraphCubic} and Theorem \ref{thm:GallaiEqualityt(G)},
\begin{equation}
    \label{eq:t(G)minimum} t(G)=\min_{J}|V_3(J)|.
\end{equation}
By counting the sum of degrees in $J$, we have
\[|V_1(J)|+3|V_3(J)|=2|E(J)|.\]
Taking into account that $J$ is a spanning subgraph, we have
\[|V_1(J)|+|V_3(J)|=|V(G)|.\]
Thus,
\[2|E(J)|=|V_1(J)|+3|V_3(J)|=|V(G)|+2|V_3(J)|,\]
or
\begin{equation}
    \label{eq:E(G)V3(G)} 2|E(J)|=|V(G)|+2|V_3(J)|.
\end{equation}
The problem of computing $t(G)$, by (\ref{eq:t(G)minimum}) is equivalent to minimization of $|V_3(J)|$. The latter is equivalent to minimization of $|E(J)|$ by (\ref{eq:E(G)V3(G)}). Thus, we need to understand the computational complexity of finding a parity subgraph with minimum number of edges in a given cubic graph $G$. The latter is polynomial time solvable as it is stated on pages 233--234 of \cite{Lovasz}. The authors of \cite{Lovasz} refer to \cite{EJ1973} as a source for the polynomial time solvability of the problem of finding a smallest parity subgraph in arbitrary graphs which may not be cubic. Let us note that in \cite{Lovasz}, the authors use the word ``join" in order to refer to parity subgraphs.

Finally, let us turn to the parameter $T(G)$ that we defined for all bridgeless cubic graphs. Note that $T(G)= 0$ is equivalent to the statement that $E(G)$ can be covered with four perfect matchings. The problem of checking this property is NP-complete in the class of bridgeless cubic graphs as it is proved in \cite{4CoverabilityNPhardness}. Note that the problem remains NP-complete in the class of cyclically 4-edge-connected cubic graphs as \cite{SkovConf4coverability} demonstrates. Recall that a connected cubic graph is cyclically 4-edge-connected if it does not contain 2-edge-cuts and all 3-edge-cuts in it are trivial.

\section*{Acknowledgement} Vahan Mkrtchyan would like to thank Hrant Khachatryan for suggesting him to work with trivial/non-trivial bridges in cubic graphs.

\bigskip

%% References


\bibliographystyle{elsarticle-num}

% \bibliographystyle{elsarticle-harv}
% \bibliographystyle{elsarticle-num-names}
% \bibliographystyle{model1a-num-names}
% \bibliographystyle{model1b-num-names}
% \bibliographystyle{model1c-num-names}
% \bibliographystyle{model2-names}
% \bibliographystyle{model3a-num-names}
% \bibliographystyle{model3-num-names}
% \bibliographystyle{model4-names}
% \bibliographystyle{model5-names}
% \bibliographystyle{model6-num-names}

%\bibliography{sample}

\section{Appendix}

For the proof of statement $(ii)$ of Theorem \ref{thm:t(G)upperboundTight}, assume that $G$ is a counterexample of minimum order. Observe that $t(G)\geq \frac{1}{6}|V(G)|$ implies $\frac{\ell(G)}{|V(G)|} \leq \frac{5}{6}$ by Theorem \ref{thm:GallaiEqualityt(G)}. Clearly, $G$ is connected. By Theorem \ref{thm:Petersen}, we can assume that $G$ contains at least three bridges. We proceed by proving a series of intermediate statements that ultimately yield a contradiction, disproving the existence of such a counterexample.

\begin{claimprime}
    \label{claim:TreeEndblocks3Simple} Every end-block of $G$ is isomorphic to $W'$.
\end{claimprime}

\begin{proof} Suppose some end-block $B$ has $h\geq 7$ vertices. Define a smaller simple cubic graph $G'$ obtained from $G$ by removing $B$ from $G$ and attaching a copy of $W'$ (Figure \ref{fig:claim1Simple}). 

\begin{figure}[ht]
  
  \begin{center}

		\begin{tikzpicture}%[-,>=stealth',shorten >=1pt,auto,node distance=3cm,
		%	thick,main node/.style={circle,fill=blue!20,draw,font=\sffamily\Large\bfseries}]
			% \begin{tikzpicture}

            %\node at (-10, -2.35) {$e$};

            %\node at (-10.5, -1.65) {$u$};
            %\node at (-9.5, -1.65) {$v$};

            %\node at (-6.5, -1.65) {$u$};
            %\node at (-4.5, -1.65) {$v$};
            %\node at (-5.5, -2.35) {$w_e$};

            %\draw (18:1cm) circle (2pt);
            \draw (-10.5,-0.5) circle (0.5cm);
            \node at (-10.5, -0.5) {$B$};
            \node at (-11.5, -0.5) {$G$};
             \node at (-6.75, -0.5) {$G'$};

            \tikzstyle{every node}=[circle, draw, fill=black!50,
                        inner sep=0pt, minimum width=4pt]
																							
			\node[circle,fill=black,draw] at (-5.5,-1) (n1) {};
			%\draw[black,fill=black,thick] (-5.5,-1) circle [radius=0.1cm] ;
			%\node at (-5, 4.35) (l1) {$r$ $(2,0)$};
			
			%\draw[fill=black,thick] (-5.5,-1) circle (2pt);
			%\node at (-5.5, -1) (n1){};
			
			%\draw[fill=black] (18:1cm) circle (2pt);
			
			%\draw[fill=black] (240:1cm) circle (2pt);
			
			%\draw[fill=black] ((-5.5,-1):1cm) circle (2pt);

                \node[circle,fill=black,draw] at (-5.5,-1) (n1) {};
                \node[circle,fill=black,draw] at (-10.5,-1) (n11) {};
																								
			\node[circle,fill=black,draw] at (-6, -0.5) (n2) {};
			% \node at (-3.3, 3.35) (l2) {$(0,1)$};
                \node[circle,fill=black,draw] at (-6, 0.5) (n22) {};
																								
			\node[circle,fill=black,draw] at (-5,-0.5) (n3) {};
			%  \node at (-6.8, 3.35) (l2) {$(0,1)$};
                \node[circle,fill=black,draw] at (-5,0.5) (n33) {};

                \node[circle,fill=black,draw] at (-5.5,-2) (n4) {};
                \node[circle,fill=black,draw] at (-10.5,-2) (n44) {};

                \node[circle,fill=black,draw] at (-6.5,-2) (n5) {};
                \node[circle,fill=black,draw] at (-11.5,-2) (n55) {};
                
                \node[circle,fill=black,draw] at (-4.5,-2) (n6) {};
                \node[circle,fill=black,draw] at (-9.5,-2) (n66) {};

                 %\node[circle,fill=black,draw] at (-10.5,-2) (n7) {};
                %\node[circle,fill=black,draw] at (-9.5,-2) (n8) {};

			\path[every node]
			(n1) edge  (n2)
                edge (n4)

			edge  (n3)
			
			%edge [bend right] (n2)
			%edge [bend left] (n2)
			%edge [bend right] (n3)
																								   	
			%(n2) edge (n3)
			%edge [bend left] (n3)
			%edge [bend left] (n3)
			%edge [bend right] (n3)
                (n4) edge (n5)
                (n4) edge (n6)

                (n22) edge (n2)
                (n22) edge (n3)

                (n22) edge (n33)
                
                (n33) edge (n3)
                (n33) edge (n2)

                %(n7) edge (n8)
                (n11) edge (n44)  
                (n44) edge (n55)
                (n44) edge (n66)

			;
		\end{tikzpicture}
																
	\end{center}
								
	\caption{Obtaining the cubic graph $G'$ from $G$.}
	\label{fig:claim1Simple}

\end{figure}
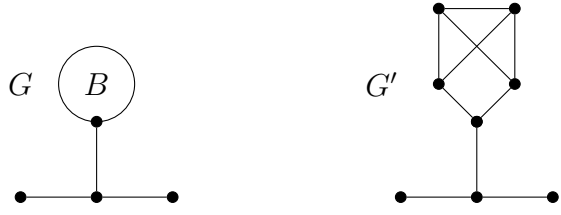

We have:
\[|V(G')|=|V(G)|-(h-5),\]
and
\[\ell(G')=\ell(G)-(h-5).\]
The latter follows from observation that the end-block has a 2-factor (see Corollary \ref{cor:Endblock2factor}), hence every maximum even subgraph covers all these $h$ vertices. Since $h\geq 7$, we have $|V(G')|<|V(G)|$. Thus,  by minimality of $G$, we have
\[\frac{\ell(G)}{|V(G)|}=\frac{\ell(G')+(h-5)}{|V(G')|+(h-5)}\geq \frac{\ell(G')}{|V(G')|}> \frac{5}{6}\]
contradicting that $G$ is a counterexample. Here we also used the fact that $\ell(G')\leq |V(G')|$.
The proof of Claim \ref{claim:TreeEndblocks3Simple} is complete.
\end{proof}

\begin{claimprime}
    \label{claim:AdjacentEndBlocksSimple} There is no vertex $w$ of $G$, such that two end-blocks of $G$ are joined to $w$ with two bridges. In other words, different end-blocks have different roots.
\end{claimprime}

\begin{proof} On the opposite assumption, assume that $w$ is incident to vertices $x$ and $y$ such that $x$ and $y$ lie in end-blocks (Figure \ref{fig:claim2Simple}). By Claim \ref{claim:TreeEndblocks3Simple}, these end-blocks have five vertices. Consider the simple cubic graph $H$ obtained from $G$ by removing these two end-blocks and attaching a copy of $W'$ to $w$ (Figure \ref{fig:claim2Simple}). %Note that $w$ becomes the degree-two vertex of the new end-block of $H$. 

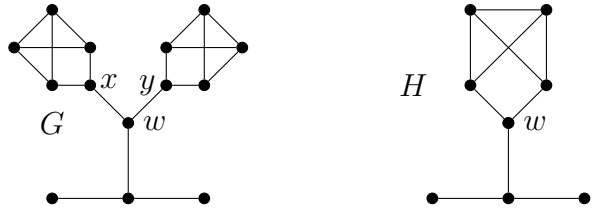
\begin{figure}[ht]
  
  \begin{center}

		\begin{tikzpicture}%[-,>=stealth',shorten >=1pt,auto,node distance=3cm,
		%	thick,main node/.style={circle,fill=blue!20,draw,font=\sffamily\Large\bfseries}]
			% \begin{tikzpicture}

            %\node at (-10, -2.35) {$e$};

            %\node at (-10.5, -1.65) {$u$};
            %\node at (-9.5, -1.65) {$v$};

            %\node at (-6.5, -1.65) {$u$};
            %\node at (-4.5, -1.65) {$v$};
            %\node at (-5.5, -2.35) {$w_e$};

            %\draw (18:1cm) circle (2pt);
            %\draw (-10.5,-0.5) circle (0.5cm);
            %\node at (-10.5, -0.5) {$B$};
            \node at (-11.5, -1) {$G$};
             \node at (-6.75, -0.5) {$H$};

             \node at (-10.15, -1) {$w$};
              \node at (-10.75, -0.45) {$x$};
              \node at (-10.25, -0.5) {$y$};
              
             \node at (-5.15, -1) {$w$};

            \tikzstyle{every node}=[circle, draw, fill=black!50,
                        inner sep=0pt, minimum width=4pt]
																							
			\node[circle,fill=black,draw] at (-5.5,-1) (n1) {};
			%\draw[black,fill=black,thick] (-5.5,-1) circle [radius=0.1cm] ;
			%\node at (-5, 4.35) (l1) {$r$ $(2,0)$};
			
			%\draw[fill=black,thick] (-5.5,-1) circle (2pt);
			%\node at (-5.5, -1) (n1){};
			
			%\draw[fill=black] (18:1cm) circle (2pt);
			
			%\draw[fill=black] (240:1cm) circle (2pt);
			
			%\draw[fill=black] ((-5.5,-1):1cm) circle (2pt);

                \node[circle,fill=black,draw] at (-5.5,-1) (n1) {};
                \node[circle,fill=black,draw] at (-10.5,-1) (n11) {};
                \node[circle,fill=black,draw] at (-10,-0.5) (n111) {};
                \node[circle,fill=black,draw] at (-10,0) (n1111) {};
                \node[circle,fill=black,draw] at (-9.5,-0.5) (n1112) {};

                \node[circle,fill=black,draw] at (-12, 0) (n11111) {};
                \node[circle,fill=black,draw] at (-11.5, 0.5) (n11112) {};

                \node[circle,fill=black,draw] at (-9, 0) (n11121) {};
                \node[circle,fill=black,draw] at (-9.5, 0.5) (n11122) {};
                
                \node[circle,fill=black,draw] at (-11,-0.5) (n112) {};
                \node[circle,fill=black,draw] at (-11.5,-0.5) (n1121) {};
                \node[circle,fill=black,draw] at (-11,0) (n1122) {};
                
			\node[circle,fill=black,draw] at (-6, -0.5) (n2) {};
                \node[circle,fill=black,draw] at (-6, 0.5) (n22) {};
			% \node at (-3.3, 3.35) (l2) {$(0,1)$};
																								
			\node[circle,fill=black,draw] at (-5,-0.5) (n3) {};
                \node[circle,fill=black,draw] at (-5, 0.5) (n33) {};
			%  \node at (-6.8, 3.35) (l2) {$(0,1)$};

                \node[circle,fill=black,draw] at (-5.5,-2) (n4) {};
                \node[circle,fill=black,draw] at (-10.5,-2) (n44) {};

                \node[circle,fill=black,draw] at (-6.5,-2) (n5) {};
                \node[circle,fill=black,draw] at (-11.5,-2) (n55) {};
                
                \node[circle,fill=black,draw] at (-4.5,-2) (n6) {};
                \node[circle,fill=black,draw] at (-9.5,-2) (n66) {};

                 %\node[circle,fill=black,draw] at (-10.5,-2) (n7) {};
                %\node[circle,fill=black,draw] at (-9.5,-2) (n8) {};

			\path[every node]
			(n1) edge  (n2)
                edge (n4)

			edge  (n3)
			
			%edge [bend right] (n2)
			%edge [bend left] (n2)
			%edge [bend right] (n3)
																								   	
			%(n2) edge (n3)
			%edge [bend left] (n3)
            
			%edge [bend left] (n3)
			%edge [bend right] (n3)
                (n4) edge (n5)
                (n4) edge (n6)

                %(n7) edge (n8)
                (n11) edge (n44)  
                (n44) edge (n55)
                (n44) edge (n66)

                (n11) edge (n111)
                (n11) edge (n112)

                (n111) edge (n1111) 
                (n111) edge (n1112) 
                
                %(n1111) edge (n1112)
                %edge [bend left] (n1112)

                %(n1121) edge (n1122)
                %edge [bend left] (n1122)

                (n112) edge (n1121)
                edge (n1122)

                %(n1111) edge (n11111)

                (n11111) edge (n11112)
                        
                (n11111) edge (n1121)

                %(n) edge (n1121)
                (n11112) edge (n1122)

                (n11121) edge (n11122)

                (n1111) edge (n11121)

                (n1112) edge (n11122)

                %(n1111) edge (n11121)
                (n1112) edge (n11121)
                (n11122) edge (n1111)

                (n11112) edge (n1121)
                (n11111) edge (n1122)

                (n2) edge (n22)
                (n2) edge (n33)

                (n3) edge (n22)
                (n3) edge (n33)

                (n22) edge (n33)

			;
		\end{tikzpicture}
																
	\end{center}
								
	\caption{Obtaining the cubic graph $H$ from $G$.}
	\label{fig:claim2Simple}

\end{figure}

We have
\[|V(G)|=|V(H)|+6,\]
and
\[\ell(G)=\ell(H)+5.\]
Since $|V(H)|<|V(G)|$, by minimality of $G$, we have
\[\frac{\ell(H)}{|V(H)|}> \frac{5}{6}.\]
Hence, 
\[\frac{\ell(G)}{|V(G)|}=\frac{\ell(H)+5}{|V(H)|+6}> \frac{5}{6},\]
which contradicts our choice of $G$ as a counterexample. The proof of Claim \ref{claim:AdjacentEndBlocksSimple} is complete.
\end{proof}

Our next claim states that roots of end-blocks form an independent set.

\begin{claimprime}
    \label{claim:AdjacentRootsEndBlocksSimple} There are no two roots $w_1$ and $w_2$ of end-blocks in $G$, such that $w_1w_2\in E(G)$.
\end{claimprime}

\begin{proof} Suppose two roots $w_1$ and $w_2$ are adjacent in $G$ (Figure \ref{fig:claim3Simple}). 
First of all, let us observe that $w_1w_2$ cannot be a double edge since $G$ has more than two bridges. Moreover, by previous claim there are no two end-blocks of $G$ joined in $w_2$. Then, we denote by $w_3$ the neighbour of $w_2$ distinct from $w_1$ and such that $w_2w_3$ is not the bridge in an end-block (Figure \ref{fig:claim3Simple}).

\begin{figure}[ht]
  
  \begin{center}

		\begin{tikzpicture}%[-,>=stealth',shorten >=1pt,auto,node distance=3cm,
		%	thick,main node/.style={circle,fill=blue!20,draw,font=\sffamily\Large\bfseries}]
			% \begin{tikzpicture}

            %\node at (-10, -2.35) {$e$};

            %\node at (-10.5, -1.65) {$u$};
            %\node at (-9.5, -1.65) {$v$};

            %\node at (-6.5, -1.65) {$u$};
            %\node at (-4.5, -1.65) {$v$};
            %\node at (-5.5, -2.35) {$w_e$};

            %\draw (18:1cm) circle (2pt);
           % \draw (-10.5,-0.5) circle (0.5cm);
            %\node at (-10.5, -0.5) {$B$};
            %\node at (-11.5, -0.5) {$G$};
             %\node at (-6.75, -0.5) {$G'$};

             \node at (-9, -2) {$G$}; \node at (-0.5, -2) {$H$};

             \node at (-8, -2) {$w_1$}; \node at (-2.35, -1.9) {$w_1$};
             \node at (-5, -2) {$w_2$};
             \node at (-4,-3) {$w_3$}; \node at (0.35,-3) {$w_3$};
             \draw (-7.5,-2) -- (-9,-3) ;
             \draw (-5.5,-2) -- (-4.5,-3) ;
             \draw (-2,-2)--(-3.5,-3);

             \draw[dashed] (-4.5,-3) -- (-7.5,-2);

            \tikzstyle{every node}=[circle, draw, fill=black!50,
                        inner sep=0pt, minimum width=4pt]
																							
			%\node[circle,fill=black,draw] at (-5.5,-1) (n1) {};
			%\draw[black,fill=black,thick] (-5.5,-1) circle [radius=0.1cm] ;
			%\node at (-5, 4.35) (l1) {$r$ $(2,0)$};
			
			%\draw[fill=black,thick] (-5.5,-1) circle (2pt);
			%\node at (-5.5, -1) (n1){};
			
			%\draw[fill=black] (18:1cm) circle (2pt);
			
			%\draw[fill=black] (240:1cm) circle (2pt);
			
			%\draw[fill=black] ((-5.5,-1):1cm) circle (2pt);

                \node[circle,fill=black,draw] at (-5.5,-1) (n1) {};
                \node[circle,fill=black,draw] at (-7.5,-1) (n11) {};
                \node[circle,fill=black,draw] at (-4.5,-3) (n111) {};
            %    \node[circle,fill=black,draw] at (-10.5,-1) (n11) {};

            \node[circle,fill=black,draw] at (-2,-2) (n1H) {};
            \node[circle,fill=black,draw] at (0,-3) (n3H) {};
            \node[circle,fill=black,draw] at (-2,-1) (w1H) {};
            \node[circle,fill=black,draw] at (-2.5, -0.5) (w11H) {};
             \node[circle,fill=black,draw] at (-2.5, 0.5) (w111H) {};
            
            \node[circle,fill=black,draw] at (-1.5,-0.5) (w12H) {};
            \node[circle,fill=black,draw] at (-1.5, 0.5) (w121H) {};

			\node[circle,fill=black,draw] at (-6, -0.5) (n2) {};
                \node[circle,fill=black,draw] at (-6, 0.5) (n21) {};
                
                \node[circle,fill=black,draw] at (-8, -0.5) (n22) {};
                 \node[circle,fill=black,draw] at (-8, 0.5) (n221) {};
			% \node at (-3.3, 3.35) (l2) {$(0,1)$};
																								
			\node[circle,fill=black,draw] at (-5,-0.5) (n3) {};
            \node[circle,fill=black,draw] at (-5, 0.5) (n31) {};
                \node[circle,fill=black,draw] at (-7,-0.5) (n33) {};
                \node[circle,fill=black,draw] at (-7, 0.5) (n331) {};
			%  \node at (-6.8, 3.35) (l2) {$(0,1)$};

                \node[circle,fill=black,draw] at (-5.5,-2) (n4) {};
                \node[circle,fill=black,draw] at (-7.5,-2) (n44) {};
             %   \node[circle,fill=black,draw] at (-10.5,-2) (n44) {};

               % \node[circle,fill=black,draw] at (-6.5,-2) (n5) {};
              %  \node[circle,fill=black,draw] at (-11.5,-2) (n55) {};
                
               % \node[circle,fill=black,draw] at (-4.5,-2) (n6) {};
               % \node[circle,fill=black,draw] at (-9.5,-2) (n66) {};

                 %\node[circle,fill=black,draw] at (-10.5,-2) (n7) {};
                %\node[circle,fill=black,draw] at (-9.5,-2) (n8) {};

			\path[every node]
                 (w1H) edge (w11H)
                  (w1H) edge (w12H)
                
                %(w11H) edge (w12H)
                %edge [bend left] (w12H)
                
                (w1H) edge (n1H)
                (n1H) edge (n3H)
			(n1) edge  (n2)
                edge (n4)

			edge  (n3)

                (n2) edge (n21)
                (n2) edge (n31)
                (n3) edge (n21)
                (n3) edge (n31)
                (n21) edge (n31)
			
			%edge [bend right] (n2)
			%edge [bend left] (n2)
			%edge [bend right] (n3)
																								   	
			%(n2) edge (n3)
			%edge [bend left] (n3)
			%edge [bend left] (n3)
			%edge [bend right] (n3)
                %(n4) edge (n5)
                %(n4) edge (n6)
                (n4) edge (n44)

                %(n7) edge (n8)
                
				(n44) edge (n11)

                (n11) edge (n22)
                (n11) edge (n33)
                %(n22) edge (n33)
                %edge [bend left] (n33)

                (w11H) edge (w111H)
                (w11H) edge (w121H)
                (w12H) edge (w111H)
                (w12H) edge (w121H)
                (w111H) edge (w121H)

                (n22) edge (n221)
                (n22) edge (n331)
                (n33) edge (n221)
                (n33) edge (n331)
                (n221) edge (n331)

			;
		\end{tikzpicture}
																
	\end{center}
								
	\caption{Obtaining the cubic graph $H$ from $G$.}
	\label{fig:claim3Simple}

\end{figure}
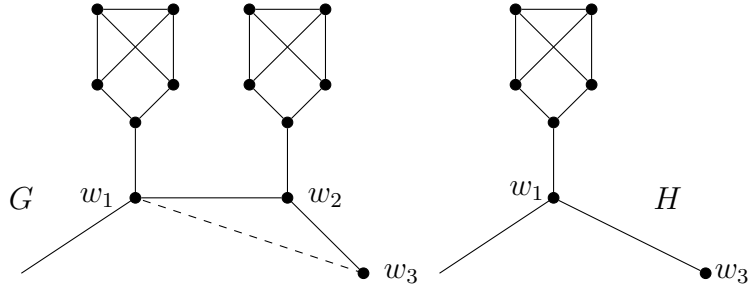

We consider two cases.

Case 1: $w_1w_3\notin E(G)$. Consider the cubic graph $H$ obtained from $G$ by removing the end-block having $w_2$ as a root and adding a new edge $w_1w_3$ (Figure \ref{fig:claim3Simple}). Note that $w_1w_3$ is not a double edge in $H$. Hence, $H$ is a simple cubic graph.

We have:
\[|V(G)|=|V(H)|+6,\]
and
\[\ell(G)\geq \ell(H)+5.\]
Since $|V(H)|<|V(G)|$, we have
\[\frac{\ell(H)}{|V(H)|}> \frac{5}{6}\]
by minimality of $G$. Hence
\[\frac{\ell(G)}{|V(G)|}\geq \frac{\ell(H)+5}{|V(H)|+6} > \frac{5}{6},\]
which contradicts our choice of $G$ as a counterexample. 

Case 2: $w_1w_3\in E(G)$. Note that $w_1, w_2, w_3$ form a triangle $K$ in $G$. Moreover, $H=G/K$ is a simple cubic graph with $|V(H)|<|V(G)|$. Thus,
\[\frac{\ell(H)}{|V(H)|}>\frac{5}{6}.\]
Since $|V(G)|=|V(H)|+2$ and $\ell(G)=\ell(H)+3$, we get
\[\frac{\ell(G)}{|V(G)|}=\frac{\ell(H)+3}{|V(H)|+2}>\frac{\ell(H)+2}{|V(H)|+2}\geq \frac{\ell(H)}{|V(H)|}>\frac{5}{6}.\]
We used the trivial inequality $\ell(H)\leq |V(H)|$ above. The proof of Claim \ref{claim:AdjacentRootsEndBlocksSimple} is complete.
\end{proof}

The proved claims allow us to view our counterexample $G$ as one obtained from a cubic graph $G_0$ by taking a subset $E_0\subseteq E(G_0)$, subdividing edges of $E_0$ once and attaching a copy of $W'$ to them. %(see Lemma \ref{lem:BridgelessMatchingSubdivide})

Let us say that a bridge $e$ of $G$ is trivial if \[\min\{|V(G_1)|, |V(G_2)|\}=5,\] where $G_1$ and $G_2$ denote the components of $G-e$.

\begin{claimprime}
    \label{claim:NonTrivialBridgeSimple} $G$ contains a non-trivial bridge.
\end{claimprime}

\begin{proof} Assume that all bridges in $G$ are trivial. Then the graph $G_0$ discussed above is a bridgeless cubic graph. By Lemma \ref{lem:BridgelessMatchingSubdivide},
\[t(G)\leq |E_0|\leq \frac{|V(G)|-|V(G_0)|}{6} < \frac{|V(G)|}{6},\]
since for every edge $e\in E_0$ we have six vertices in $G$. This contradicts the assumption that $G$ is a counterexample. The proof of Claim \ref{claim:NonTrivialBridgeSimple} is complete.
\end{proof}

We are ready to complete the proof of statement $(ii)$. Let $f$ be a non-trivial bridge in $G$. Let $G_1$ and $G_2$ be the components of $G-f$. Consider two simple cubic graphs $H_1$ and $H_2$ obtained from $G_1$ and $G_2$, respectively, by attaching copies of $W'$ to the end-vertices of $f$ (Figure \ref{fig:proofendSimple}). 

\begin{figure}[ht]
  
  \begin{center}

		\begin{tikzpicture}%[-,>=stealth',shorten >=1pt,auto,node distance=3cm,
		%	thick,main node/.style={circle,fill=blue!20,draw,font=\sffamily\Large\bfseries}]
			% \begin{tikzpicture}

            %\node at (-10, -2.35) {$e$};

            %\node at (-10.5, -1.65) {$u$};
            %\node at (-9.5, -1.65) {$v$};

            %\node at (-6.5, -1.65) {$u$};
            %\node at (-4.5, -1.65) {$v$};
            %\node at (-5.5, -2.35) {$w_e$};

            %\draw (18:1cm) circle (2pt);
            \draw (-10.5,-0.5) circle (0.5cm);
            \draw (-7.5,-0.5) circle (0.5cm);
            \node at (-10.5, -0.5) {$G_1$}; 
            \node at (-4.5, -0.5) {$G_1$}; 
            \draw (-4.5,-0.5) circle (0.5cm);

            \node at (-1, -0.5) {$H_1$};
            \node at (-1, -2) {$H_2$};
             \draw (-2.5,-2) circle (0.5cm);
             \node at (-2.5, -2) {$G_2$}; 
            
            \node at (-7.5, -0.5) {$G_2$};

            \node at (-11.5, -0.5) {$G$};
             \node at (-9, -0.15) {$f$};
            % \node at (-6.75, -0.5) {$G'$};
            \draw (-10,-0.5)--(-8,-0.5);

            \tikzstyle{every node}=[circle, draw, fill=black!50,
                        inner sep=0pt, minimum width=4pt]
																							
			\node[circle,fill=black,draw] at (-4,-0.5) (n1) {};
                 \node[circle,fill=black,draw] at (-3,-0.5) (n2) {};
                 \node[circle,fill=black,draw] at (-2.5,0) (n21) {};
                 \node[circle,fill=black,draw] at (-1.5,0) (n211) {};
                 
                 \node[circle,fill=black,draw] at (-2.5,-1) (n22) {};
                 \node[circle,fill=black,draw] at (-1.5,-1) (n221) {};

                 \node[circle,fill=black,draw] at (-3,-2) (n1a) {};
                 \node[circle,fill=black,draw] at (-4,-2) (n2a) {};
                  \node[circle,fill=black,draw] at (-4.5,-1.5) (n21a) {};
                  \node[circle,fill=black,draw] at (-5.5,-1.5) (n21b) {};
                  
                  \node[circle,fill=black,draw] at (-4.5,-2.5) (n22a) {};
                  \node[circle,fill=black,draw] at (-5.5,-2.5) (n22b) {};
			%\draw[black,fill=black,thick] (-5.5,-1) circle [radius=0.1cm] ;
			%\node at (-5, 4.35) (l1) {$r$ $(2,0)$};
			
			%\draw[fill=black,thick] (-5.5,-1) circle (2pt);
			%\node at (-5.5, -1) (n1){};
			
			%\draw[fill=black] (18:1cm) circle (2pt);
			
			%\draw[fill=black] (240:1cm) circle (2pt);
			
			%\draw[fill=black] ((-5.5,-1):1cm) circle (2pt);

               % \node[circle,fill=black,draw] at (-5.5,-1) (n1) {};
              %  \node[circle,fill=black,draw] at (-10.5,-1) (n11) {};
																								
			%\node[circle,fill=black,draw] at (-6, -0.5) (n2) {};
			% \node at (-3.3, 3.35) (l2) {$(0,1)$};
																								
			%\node[circle,fill=black,draw] at (-5,-0.5) (n3) {};
			%  \node at (-6.8, 3.35) (l2) {$(0,1)$};

                %\node[circle,fill=black,draw] at (-5.5,-2) (n4) {};
               % \node[circle,fill=black,draw] at (-10.5,-2) (n44) {};

                %\node[circle,fill=black,draw] at (-6.5,-2) (n5) {};
               % \node[circle,fill=black,draw] at (-11.5,-2) (n55) {};
                
                %\node[circle,fill=black,draw] at (-4.5,-2) (n6) {};
               % \node[circle,fill=black,draw] at (-9.5,-2) (n66) {};

                 %\node[circle,fill=black,draw] at (-10.5,-2) (n7) {};
                %\node[circle,fill=black,draw] at (-9.5,-2) (n8) {};

			\path[every node]
			(n1) edge  (n2)
                (n2) edge (n21)
                (n2) edge (n22)
               % (n21) edge (n22)
                %edge [bend left] (n22)

                (n1a) edge (n2a)

                (n2a) edge (n21a)
                (n2a) edge (n22a)
                %(n21a) edge (n22a)
                %edge [bend right] (n22a)

			%edge  (n3)
			
			%edge [bend right] (n2)
			%edge [bend left] (n2)
			%edge [bend right] (n3)
																								   	
			%(n2) edge (n3)
			%edge [bend left] (n3)
			%edge [bend left] (n3)
			%edge [bend right] (n3)
                %(n4) edge (n5)
                %(n4) edge (n6)

                (n21) edge (n211)
                (n21) edge (n221)
                (n22) edge (n211)
                (n22) edge (n221)
                (n211) edge (n221)

                (n21a) edge (n21b)
                (n21a) edge (n22b)
                (n22a) edge (n21b)
                (n22a) edge (n22b)
                (n21b) edge (n22b)

			;
		\end{tikzpicture}
																
	\end{center}
								
	\caption{Obtaining the cubic graphs $H_1$ and $H_2$ from $G$.}
	\label{fig:proofendSimple}

\end{figure}
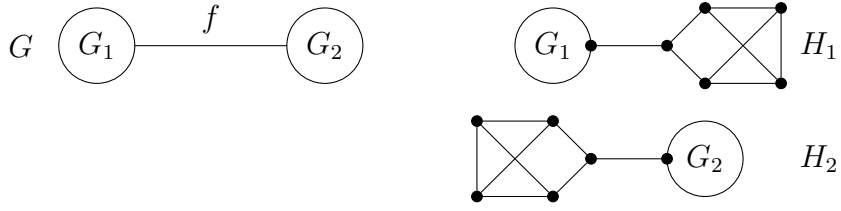

Since $f$ is non-trivial, we have $|V(H_1)|, |V(H_2)|<|V(G)|$. Note that we can always choose the non-trivial bridge $f$ such that

\begin{enumerate}
    \item [(a)] $H_1$ contains at least two end-blocks,

    \item [(b)] all bridges in $H_1$ are trivial.
\end{enumerate}

We have:
\[|V(H_1)|=|V(G_1)|+5,|V(H_2)|=|V(G_2)|+5,\]
and since $f$ is a bridge
\[\ell(H_1)=\ell(G_1)+5,\ell(H_2)=\ell(G_2)+5.\]
Since $|V(H_2)|<|V(G)|$, we have
\[\ell(G_2)+5=\ell(H_2)> \frac{5}{6}|V(H_2)|=\frac{5}{6} \left(|V(G_2)|+5\right)\]
by minimality of $G$. This is equivalent to
\begin{equation}
    \label{eq:G2Simple} \ell(G_2)> \frac{5}{6} |V(G_2)|-\frac{5}{6}.
\end{equation} On the other hand, since in $H_1$ all bridges are trivial, by Lemma \ref{lem:E2} and minimality of $G$
\[t(H_1)\leq |E_0(H_1)|-2\leq \frac{|V(H_1)|}{6}-2<\frac{|V(H_1)|}{6}-\frac{5}{3}.\]
By Theorem \ref{thm:GallaiEqualityt(G)}, the latter means
\[\ell(H_1) >\frac{5|V(H_1)|}{6}+\frac{5}{3}.\]
Since $\ell(H_1)=\ell(G_1)+5$ and $|V(H_1)|=|V(G_1)|+5$, the last inequality is equivalent to
\begin{equation}
    \label{eq:G1Simple} \ell(G_1) >  \frac{5}{6}|V(G_1)|+\frac{5}{6}.
\end{equation} (\ref{eq:G2Simple}) and (\ref{eq:G1Simple}) together imply
\[\ell(G)=\ell(G_1)+\ell(G_2)> \frac{5}{6}|V(G_1)|+\frac{5}{6}+\frac{5}{6}|V(G_2)|-\frac{5}{6}=\frac{5}{6}|V(G)|\]
which contradicts our assumption that $G$ is a counterexample. In the last equation, we used equalities
\[|V(G)|=|V(G_1)|+|V(G_2)|,\]
and
\[\ell(G)=\ell(G_1)+\ell(G_2).\]
The latter follows from our choice of $f$ as a bridge. Proof of $(ii)$ is complete.

\end{document}